\documentclass[11pt]{amsart}
\numberwithin{equation}{section}

\usepackage{amsmath,amssymb,amscd,amsxtra,latexsym,
graphicx,epsfig,euscript,amsthm,mathscinet}
\usepackage{amsmath,amsfonts,amsthm,amsopn,amssymb,enumitem}
\usepackage{graphicx}
\usepackage{float}
\usepackage{a4wide}
\usepackage{color}
\usepackage{esint}
\usepackage{palatino}
\usepackage[backend=bibtex,style=numeric,sorting=nyt,
giveninits=true,isbn=false,url=false,doi=false]{biblatex}
\renewbibmacro{in:}{\ifentrytype{article}{}{\printtext{\bibstring{in}\intitlepunct}}}
\DeclareFieldFormat*{title}{#1}

\usepackage{mathrsfs}
\usepackage[colorlinks]{hyperref}
\AtBeginDocument{%
  \hypersetup{%
    linkcolor=blue,%
    citecolor=green,%
  }%
}
\usepackage{cleveref}
\crefname{equation}{}{}
\crefname{figure}{}{}

\makeatletter
\@namedef{subjclassname@2020}{\textup{2020} Mathematics Subject Classification}
\makeatother

\newtheorem{theorem}{Theorem}[section]
\newtheorem{lemma}[theorem]{Lemma}

\newtheorem{definition}[theorem]{Definition}

\newtheorem{remark}[theorem]{Remark}

\begin{document}

\title[Boundary pointwise regularity on $C^{1,\mathrm{Dini}}$ domains]{Boundary pointwise regularity for the $p$-Laplace equation on $C^{1,\mathrm{Dini}}$ domains}

\author{Xuemei Li}
\address{(X. ~Li)
  School of Mathematical Sciences, University of Jinan, Jinan, P.R.China 250022}
\email{sms\_lixm@ujn.edu.cn}

\author{Yuanyuan Lian}
\address{(Y.~Lian)
  Departamento de Geometr\'{i}a y Topolog\'{i}a, Departamento de An\'alisis matem\'atico, Universidad de Granada,   Campus Fuentenueva, 18071 Granada, Spain}
\email{lianyuanyuan.hthk@gmail.com; yuanyuanlian@correo.ugr.es}

\thanks{{\bf Acknowledgements.}
X. L. has been supported by National Natural Science Foundation of China (No. 12401257) and the Natural Science Foundation of Shandong Province, China (No. ZR2024QA045).
Y. L. has been supported by the Grants PID2020-117868GB-I00 and PID2023-150727NB-I00 of the MICIN/AEI}

\keywords{Boundary Lipschitz regularity; Hopf lemma; boundary $C^1$ regularity; $C^{1,\mathrm{Dini}}$ domains; $p$-Laplace equations.}

\subjclass[2020]{35B65; 35J25; 35J92.}

\maketitle

\noindent

\noindent
\begin{abstract}
  This paper investigates the boundary pointwise regularity of solutions to the p-Laplace equation. By treating curved boundaries as perturbations of flat ones, we establish the boundary pointwise Lipschitz regularity under the exterior $C^{1,\mathrm{Dini}}$ condition and the Hopf lemma under the interior $C^{1,\mathrm{Dini}}$ condition. We also establish the boundary pointwise $C^1$ regularity under the $C^{1,\mathrm{Dini}}$ condition by a compactness argument.
\end{abstract}

\section{Introduction}
\label{section 1}

In this paper, we study the boundary pointwise regularity of solutions to the $p$-Laplace equation
\begin{equation}\label{e.p-lap}
\begin{cases}
	\mathrm{div}\left(|D u|^{p-2} D u\right) = f & \text{in } \Omega \cap B_1, \\
	u = g & \text{on } \partial\Omega \cap B_1,
\end{cases}
\end{equation}
where $\Omega\subset \mathbb{R}^n$ is a bounded domain and $2\leq p<+\infty$. More precisely, we investigate the regularity of solutions near a boundary point $x_0$ under three different assumptions: the exterior $C^{1,\mathrm{Dini}}$ condition, the interior $C^{1,\mathrm{Dini}}$ condition and the $C^{1,\mathrm{Dini}}$ condition (see \Cref{D_C1Dini}).

\medskip

The $p$-Laplacian is the Euler-Lagrange equation associated with the $p$-Dirichlet energy and arises naturally in various nonlinear models, including non-Newtonian fluids, glaciology, and porous media. Unlike the classical Laplace operator, the $p$-Laplacian is degenerate when $p>2$ and singular when $1<p<2$. As a consequence, classical linear techniques are generally not applicable, and one cannot expect $C^2$ regularity in general.

An important direction in regularity theory is to determine the optimal smoothness, typically $C^{1,\alpha}$, that solutions retain despite degeneracy or singularity.
For interior $C^{1,\alpha}$ regularity, Ural'tseva \cite{MR0244628} made an early foundational contribution to the regularity theory of degenerate elliptic equations. Uhlenbeck \cite{MR474389} established gradient H\"{o}lder regularity for a class of nonlinear elliptic systems, including $p$-harmonic systems in the degenerate range $p \geq 2$. Later, Evans \cite{MR672713} gave an elegant new proof for the interior $C^{1,\alpha}$ regularity when $p \geq 2$. Lewis \cite{MR721568} proved the interior $C^{1,\alpha}$ regularity in the singular range $1 < p < 2$. DiBenedetto \cite{MR709038} and Tolksdorf \cite{MR727034} extended these results to more general quasilinear elliptic equations with $p$-growth structures.


For boundary regularity, Lieberman \cite{MR969499} proved global $C^{1,\alpha}$ regularity for both Dirichlet and Neumann problems, under suitable smoothness assumptions on the boundary. For rough boundaries, Kilpel\"{a}inen and Mal\'{y} \cite{MR1264000} established a Wiener criterion for regular boundary points of quasilinear equations of $p$-Laplace type, formulated in terms of the $p$-capacity of the complement. Lee, Lian, Yun, and Zhang \cite{LLYZ2025} established boundary pointwise and global $C^{1,\alpha}$ regularity for viscosity solutions of parabolic $p$-Laplace-type equations.


\medskip

In this paper, we study boundary pointwise Lipschitz regularity, the Hopf lemma and boundary pointwise $C^1$ regularity for the $p$-Laplacian. The core question is how weak the assumptions on the right-hand side $f$ and the boundary $\partial \Omega$ can be while guaranteeing such boundary regularity.

Regarding global Lipschitz regularity, Cianchi and Maz'ya \cite{MR2763349, MR3165732} established sharp rearrangement estimates and global gradient bounds for nonlinear Dirichlet and Neumann problems. In addition, over the last two decades, nonlinear potential theory has led to significant progress in the regularity theory of $p$-Laplace type equations. Duzaar and Mingione \cite{MR2823872} established pointwise gradient estimates in terms of nonlinear Wolff-type potentials. Later, Kuusi and Mingione \cite{MR3004772} obtained pointwise gradient bounds in terms of the classical first-order Riesz potential and derived criteria for local Lipschitz continuity and gradient continuity. We refer to Kuusi and Mingione \cite{MR3174278} for a comprehensive account of these potential estimates and their applications. Dong and Zhu \cite{MR4396694, MR4768412} obtained gradient estimates and gradient-continuity criteria for singular parabolic and elliptic $p$-Laplace type equations with measure data, using suitable parabolic and elliptic Riesz potentials.

\medskip

We prove the boundary pointwise Lipschitz regularity, the Hopf lemma and the boundary pointwise $C^1$ regularity by treating curved boundaries as perturbations of flat ones. A key observation is that solutions of the homogeneous $p$-Laplace equation belong to suitable Pucci classes, which allows us to apply the boundary $C^{1,\alpha}$ regularity theory on flat domains. The boundary pointwise Lipschitz regularity and the Hopf lemma are established through an iterative perturbation argument in the spirit of \cite{MR4713521} and its parabolic counterpart \cite{DLL_2025}, whereas the boundary pointwise $C^1$ regularity is obtained by a compactness method. The combination of perturbation arguments and compactness techniques is widely used in regularity theory (see, for instance, \cite{MR4088470, lian2020pointwise}).

\medskip

Before stating our main results, we introduce some notation and preliminary notions. In order to derive our regularity results, we need to consider a more general equation:
\begin{equation}\label{e.p.gene}
\mathrm{div}\left(|Du+\xi|^{p-2}(Du+\xi) \right)=f \quad  ~~\mbox{ in }~~  \Omega,
\end{equation}
where $\xi \in \mathbb{R}^n$.
First, let us recall the definitions of viscosity and weak solutions for the $p$-Laplace equation; see \cite[Definition 2.3]{MR1871417} and \cite[Definition 2.1]{MR2915869} for more details.
\begin{definition}\label{V}
Let $u, f\in C(\Omega)$. We say that $u$ is a viscosity subsolution (resp. supersolution) of \eqref{e.p.gene} when the following condition holds: if $x_0\in \Omega$, $\varphi\in C^2(\Omega)$ and $u-\varphi$ attains a local maximum (resp. minimum) at $x_0$, then
\begin{equation*}
  \mathrm{div}\left(|D \varphi+\xi|^{p-2} (D \varphi+\xi) \right) (x_0)\geq f(x_0) \ \quad ~(\mbox{resp. }~~ \mathrm{div}\left(|D \varphi+\xi|^{p-2} (D \varphi+\xi) \right) (x_0) \leq f(x_0)).
\end{equation*}
We say that $u$ is a viscosity solution of \eqref{e.p.gene} when it is both a viscosity subsolution and a viscosity supersolution.
\end{definition}

\begin{definition}\label{weak_solutions}
Let $f \in C(\Omega)$. A function $u \in W_{\mathrm{loc}}^{1,p}(\Omega)$ is a weak subsolution (resp. supersolution) to \eqref{e.p.gene} in $\Omega$ if
\begin{equation*}\label{e.weak}
\begin{aligned}
   & \int_{\Omega} |Du(x)+\xi|^{p-2}(Du(x)+\xi) \cdot D\psi(x) \, dx \leq -\int_{\Omega} \psi(x)f(x) \, dx\\
    \ \quad ~(\mbox{resp. }~~
   & \int_{\Omega} |Du(x)+\xi|^{p-2}(Du(x)+\xi) \cdot D\psi(x) \, dx \geq -\int_{\Omega} \psi(x)f(x) \, dx)
\end{aligned}
\end{equation*}
for every nonnegative $\psi \in C_0^\infty(\Omega)$.

A weak solution is both a subsolution and a supersolution.
\end{definition}

\begin{remark}\label{vis_weak}
Viscosity solutions are equivalent to weak solutions; see \cite{MR2915869, MR1871417, MR3918385}. We will work with both viscosity and weak solutions throughout this paper.
\end{remark}

Next, we recall the definitions of the Pucci extremal operators and the associated Pucci classes.
\begin{definition}\label{d-Sf}
Let $\mathscr{S}^n$ denote the space of real $n\times n$ symmetric matrices. For $M\in \mathscr{S}^n$, let $\mu_i$ $(1\leq i\leq n)$ be the eigenvalues of $M$. For $0<\lambda\leq \Lambda$, the Pucci extremal operators are defined by
\begin{equation*}
\begin{aligned}
  &\mathcal{M}^+(M,\lambda,\Lambda)=\Lambda \sum_{\mu_i>0}\mu_i 
+\lambda \sum_{\mu_i<0}\mu_i, \\
  &\mathcal{M}^-(M,\lambda,\Lambda)=\lambda \sum_{\mu_i>0}\mu_i 
+\Lambda \sum_{\mu_i<0}\mu_i.
\end{aligned}
\end{equation*}

We define the corresponding Pucci classes as follows. We say that $u\in \underline{\mathcal S}(\lambda,\Lambda,f)$ if $u$ is an $L^n$-viscosity subsolution (see \cite{MR1376656}, also \cite[Definition 1.9]{MR4713521}) of
\begin{equation*}
  \mathcal{M}^+(D^2u,\lambda,\Lambda)= f.
\end{equation*}
Similarly, we denote $u\in \overline{\mathcal S}(\lambda,\Lambda,f)$ if $u$ is an $L^n$-viscosity supersolution of
\begin{equation*}
  \mathcal{M}^-(D^2u,\lambda,\Lambda)= f.
\end{equation*}
We further define
\begin{equation}\label{SC2Sf}
\begin{aligned}
  &\mathcal S(\lambda,\Lambda,f)=\underline{\mathcal S}(\lambda,\Lambda,f)\cap \overline{\mathcal S}(\lambda,\Lambda,f).\\
\end{aligned}
\end{equation}

\end{definition}

\medskip

We now fix some notation. Let $e_i=(0,\ldots,0,1,\ldots,0)$ be the $i$-th standard coordinate vector in $\mathbb{R}^n$. We write $x'=(x_1,\ldots,x_{n-1})$ and $x=(x',x_n)$. Set
\begin{equation*}
  B_r(x_0)=\{x\in \mathbb{R}^n: |x-x_0|<r\}, \quad
  B^+_r (x_0) = B_r(x_0) \cap  \{x\in \mathbb{R}^n:x_{n}>x_{0,n}\},
\end{equation*}
where $x_{0,n}$ denotes the $n$-th coordinate of $x_0$.
Moreover, let
\begin{equation*}
  B_r'(x_0)=\{x'\in \mathbb{R}^{n-1}: |x'-x'_0|<r\},\quad T_r(x_0)= \{(x',x_{0,n}): x'\in B_r'(x_0)\},
\end{equation*}
and
\begin{equation*}
  \Omega_r (x_0)= \Omega \cap B_r(x_0), \quad
  \Omega_r^+ (x_0)=\Omega \cap B_r^+ (x_0), \quad
   (\partial\Omega)_r (x_0)= \partial\Omega\cap B_r(x_0).
\end{equation*}
When $x_0=0$, we omit $x_0$ from the notation. Additionally, we use the Einstein summation convention throughout the paper.

For a bounded domain $\Omega \subset \mathbb{R}^n$, we use the normalized $L^q$ norm
\begin{equation*}
  \|f\|^*_{L^q(\Omega)}:=\left(\frac{1}{|\Omega|}\int_{\Omega} |f|^q\right)^{\frac{1}{q}},\quad \forall 1\leq q<+\infty,
\end{equation*}
where $|\Omega|$ denotes the Lebesgue measure of $\Omega$. This normalization is convenient since $\|f\|^*_{L^q(\Omega)}$ has the same scaling as $f$.

We also recall the notion of a Dini function.
\begin{definition}
A function $\omega: (0,+\infty) \to [0,+\infty)$ is called a Dini function if it is nonnegative, nondecreasing, and satisfies the Dini condition
\begin{equation}\label{Dini}
\int_{0}^{r_0} \frac{\omega(r)}{r} dr < +\infty,
\end{equation}
for some $r_0 > 0$.
\end{definition}

We next introduce the definitions of the conditions on domains under which our main results are established.
\begin{definition}[$C^{1,\mathrm{Dini}}$ boundary conditions]\label{D_C1Dini}
Let $\Omega\subset \mathbb{R}^n$ be a bounded domain and let $\omega_{\Omega}$ be a Dini function. We say that $\Omega$ satisfies the exterior $C^{1,\mathrm{Dini}}$ condition at $x_0 \in \partial\Omega$ if there exist a constant $r_0 > 0$ and a coordinate system with origin at $x_0$ such that
\begin{equation}\label{e.Ex_C1Dini}
B_{r_0} \cap \{(x',x_n): x_n \leq -|x'|\omega_{\Omega}(|x'|)\} \subset B_{r_0} \cap \Omega^c.
\end{equation}

Similarly, we say that $\Omega$ satisfies the interior $C^{1,\mathrm{Dini}}$ condition at $x_0 \in \partial\Omega$ if there exist a constant $r_0 > 0$ and a coordinate system with origin at $x_0$ such that
\begin{equation}\label{e.In_C1Dini}
B_{r_0} \cap \{(x',x_n): x_n > |x'|\omega_{\Omega}(|x'|)\} \subset B_{r_0} \cap \Omega.
\end{equation}

Finally, we say that $\Omega$ satisfies the $C^{1,\mathrm{Dini}}$ condition at $x_0 \in \partial\Omega$ if $\Omega$ satisfies both the exterior and the interior $C^{1,\mathrm{Dini}}$ condition at $x_0$. In this case, we define
\begin{equation*}
  \|\partial\Omega\|_{C^{1,\mathrm{Dini}}(x_0)}
=\int_{0}^{r_0} \frac{\omega_{\Omega}(r)}{r} dr + \omega_{\Omega}(r_0).
\end{equation*}
\end{definition}

Finally, we introduce several notions of pointwise regularity for functions.

\begin{definition}[Pointwise regularity classes]\label{PW_reg}
Let $\Omega \subset \mathbb{R}^n$ be a bounded set, not necessarily a domain. We define the pointwise classes $C^{0,1}(x_0)$, $C^{1,\mathrm{Dini}}(x_0)$, and $C_q^{-1,\mathrm{Dini}}(x_0)$ as follows.

We say that $f$ is Lipschitz continuous at $x_0 \in \Omega$, or $f \in C^{0,1}(x_0)$, if there exist constants $r_0>0$ and $K>0$ such that
\begin{equation}\label{e.f.C01}
  |f(x) - f(x_0)| \leq K|x - x_0|, \quad \forall x \in \Omega\cap B_{r_0}(x_0).
\end{equation}
We define
\begin{equation*}
  [f]_{C^{0,1}(x_0)} := \inf \left\{ K | \eqref{e.f.C01} \mathrm{\ holds\ with\ } K \right\}
\end{equation*}
and
\begin{equation*}
  \|f\|_{C^{0,1}(x_0)} := \|f\|_{L^{\infty}(\Omega)} + [f]_{C^{0,1}(x_0)}.
\end{equation*}

Similarly, we say that $f$ is $C^{1,\mathrm{Dini}}$ at $x_0$, or $f \in C^{1,\mathrm{Dini}}(x_0)$, if there exist a vector $l$, a constant $r_0>0$, and a Dini function $\omega_f$ such that
\begin{equation}\label{e.f.C1Dini}
|f(x) - f(x_0) - l \cdot (x - x_0)| \leq |x - x_0|\omega_f(|x - x_0|), \quad \forall x \in \Omega\cap B_{r_0}(x_0).
\end{equation}
In this case, we denote $l$ by $\nabla f(x_0)$ and define
\begin{equation*}
  \|f\|_{C^{1}(x_0)} := |f(x_0)| + |l|, \quad
  [f]_{C^{1,\mathrm{Dini}}(x_0)} :=\int_{0}^{r_0}\frac{\omega_f(r)}{r}dr+\omega_f(r_0)
\end{equation*}
and
\begin{equation*}
  \|f\|_{C^{1,\mathrm{Dini}}(x_0)} := \|f\|_{C^{1}(x_0)} + [f]_{C^{1,\mathrm{Dini}}(x_0)}.
\end{equation*}
If $f\in C^{1,\mathrm{Dini}}(x)$ for every $x\in \Omega$, with the same $r_0$ and modulus $\omega_f$, and
\begin{equation*}
  \|f\|_{C^{1,\mathrm{Dini}}(\overline{\Omega})}:=  \sup_{x\in \Omega} \|f\|_{C^{1}(x)}+\sup_{x\in \Omega} [f]_{C^{1, \mathrm{Dini}}(x)}<+\infty,
\end{equation*}
then we say that $f\in C^{1,\mathrm{Dini}}(\overline{\Omega})$.

Finally, for $q\geq1$, we say that $f$ is $C_q^{-1,\mathrm{Dini}}$ at $x_0$, or $f \in C_q^{-1,\mathrm{Dini}}(x_0)$, if there exist a constant $r_0$ and a modulus of continuity $\omega_f$ such that
\begin{equation}\label{e.C-1Dini}
\|f\|^*_{L^q(\Omega\cap B_r(x_0))} \leq r^{-1}\omega_f (r), \quad \forall\, 0 < r < r_0
\end{equation}
and
\begin{equation*}
  \int_{0}^{r_0}\frac{\omega_f^{1/(p-1)}(r)}{r}dr<\infty.
\end{equation*}
We define
\begin{equation*}
   \|f\|_{C_q^{-1,\mathrm{Dini}}(x_0)} :=\left(\int_{0}^{r_0}\frac{\omega_f^{1/(p-1)}(r)}{r}dr\right)^{p-1}+\omega_f(r_0).
\end{equation*}
If $f \in C_q^{-1,\mathrm{Dini}}(x)$ for every $x \in \Omega$, with the same $r_0$ and modulus $\omega_f$, and
\begin{equation*}
  \|f\|_{C_q^{-1,\mathrm{Dini}}(\overline{\Omega})} := \sup_{x \in \Omega} \|f\|_{C_q^{-1,\mathrm{Dini}}(x)} < +\infty,
\end{equation*}
then we say that $f \in C_q^{-1,\mathrm{Dini}}(\overline{\Omega})$.
\end{definition}

\begin{remark}
Since our primary concern is the boundary pointwise regularity, we shall work locally near a boundary point throughout the paper. Hence, after a translation and a scaling, we may assume without loss of generality that $0\in\partial\Omega$ and $r_0=1$ in \Cref{D_C1Dini,PW_reg}.
\end{remark}

Now we are in a position to state our main results. The first one is the \emph{boundary pointwise Lipschitz regularity}.

\begin{theorem}\label{t.Lip}
Let $u$ be a viscosity solution of \eqref{e.p-lap}. Suppose that $\Omega$ satisfies the exterior $C^{1,\mathrm{Dini}}$ condition at $0\in \partial \Omega$ and
\begin{equation*}
 f \in C^{-1,\mathrm{Dini}}_q(0),\quad g \in C^{1,\mathrm{Dini}}(0)
\end{equation*}
for some  $q > n/p$, and $1/p + 1/q \leq 1$.
Then $u \in C^{0,1}(0)$, i.e.,
\begin{equation}\label{e.t.Lip}
|u(x) - u(0)| \leq C|x|\left(\|u\|_{L^\infty(\Omega \cap B_1)} + \|f\|^{1/(p-1)}_{C^{-1,\mathrm{Dini}}_q(0)} + \|g\|_{C^{1,\mathrm{Dini}}(0)}\right),
\forall~~x \in \Omega \cap B_1,
\end{equation}
where $C$ depends only on $n,p,q$, $\omega_\Omega$, $\omega_f$ and $\omega_g$.
\end{theorem}

The second result concerns the \emph{Hopf lemma}.
\begin{theorem}\label{t.Hopf}
Let $u$ be a viscosity solution of
\begin{equation*}
  \mathrm{div}\left(|Du|^{p-2}Du\right) = f\ \  \text{in } \Omega \cap B_1
\end{equation*}
with $u\geq 0$ and $u(0)=0$.
Suppose that $\Omega$ satisfies the interior $C^{1,\mathrm{Dini}}$ condition at $0\in \partial \Omega$ and $f \in C^{-1,\mathrm{Dini}}_q(0)$ with
\begin{equation}\label{e.c.Hopf}
\|f\|^{1/(p-1)}_{C^{-1,\mathrm{Dini}}_q(0)}\leq \varepsilon_0 \cdot u(e_n/2)
\end{equation}
for some $q > n/p$, and $1/p + 1/q \leq 1$, where $0<\varepsilon_0<1$ is a small constant depending only on $n,p,q$ and $\omega_\Omega$.
Then for any $l=(l_1,l_2,...,l_n)\in \mathbb{R}^n$ with $|l|=1$ and $l_n> 0$,
\begin{equation}\label{e.Hopf}
	u(rl)\geq c u(e_n/2) l_nr,\quad \forall~0<r<\eta,
\end{equation}
where $c$ depends only on $n,p,q$, $\omega_\Omega$ and $\eta$ depends also on $l$.
\end{theorem}

The following result is the \emph{boundary pointwise $C^1$ regularity}.

\begin{theorem}\label{t.C1}
Let $u$ be a viscosity solution of \eqref{e.p-lap}. Suppose that $\partial\Omega$ is $C^{1,\mathrm{Dini}}$ at $0$ and
\begin{equation*}
  f \in L^\infty(\Omega \cap B_1), \quad g \in C^{1,\mathrm{Dini}}(0).
\end{equation*}
Then $u \in C^{1 }(0)$, i.e., there exist an affine function $L$ and a modulus of continuity $\sigma$, depending only on $n,p,\omega_g$ and $\omega_{\Omega}$, such that
\begin{equation}\label{C1.e.u}
  |u(x) - L(x)| \leq C |x|\sigma(|x|) \left( \|u\|_{L^\infty(\Omega \cap B_1)} + \|f\|_{L^\infty(\Omega \cap B_1)}^{1/(p-1)} + \|g\|_{C^{1,\mathrm{Dini}}(0)} \right), \quad \forall x \in \Omega \cap B_{1/2},
\end{equation}
and
\begin{equation}\label{C1.e.Du0}
  |Du(0)|=|DL| \leq C \left( \|u\|_{L^\infty(\Omega \cap B_1)} + \|f\|^{1/(p-1)}_{L^\infty(\Omega \cap B_1)} + \|g\|_{C^{1,\mathrm{Dini}}(0)} \right),
\end{equation}
where $C$ depends only on $n, p$ and $\|\partial\Omega\|_{C^{1,\mathrm{Dini}}(0)}$.
\end{theorem}

The paper is organized as follows. In \Cref{sec:Pre}, we present some preliminaries. We show that viscosity solutions of the homogeneous $p$-Laplace equation belong to Pucci classes and establish a comparison estimate for solutions of different $p$-Laplace equations. In addition, we establish a compactness property for solutions of the $p$-Laplace equation. In \Cref{sec:Lip} and \Cref{sec:Hopf}, we prove the boundary pointwise Lipschitz regularity on exterior $C^{1,\mathrm{Dini}}$ domains (\Cref{t.Lip}) and the Hopf lemma on interior $C^{1,\mathrm{Dini}}$ domains (\Cref{t.Hopf}), respectively. Finally, \Cref{sec:C1} is devoted to the proof of the boundary pointwise $C^1$ regularity (\Cref{t.C1}).

\section{Preliminaries}\label{sec:Pre}
In this section, we first prove that viscosity solutions of the $p$-Laplace equation belong to suitable Pucci classes and establish an important comparison estimate for the subsequent iteration argument. Then we prove interior H\"{o}lder regularity and a uniform estimate up to the boundary.

\begin{lemma}\label{l.vis.Pucci}
Let $\xi\in \mathbb{R}^n$. Assume that $u$ is a viscosity solution of
\begin{equation}\label{e.vis.Pucci}
  \mathrm{div}\left(|Du+\xi|^{p-2}(Du+\xi) \right)=0 \quad  ~~\mbox{ in }~~  \Omega,
\end{equation}
then
\begin{equation*}
  u\in \mathcal S(1,p-1,0) \quad ~~\mbox{ in }~~ \Omega.
\end{equation*}
\end{lemma}


\begin{proof}
Inspired by the proof of \cite[Lemma 6]{MR2995669}, we argue as follows. We first reduce the problem to the case $\xi=0$ by introducing
\begin{equation*}
  v=u+\xi\cdot x
\end{equation*}
which satisfies
\begin{equation*}
  \mathrm{div}(|Dv|^{p-2} Dv)=0 \quad ~~\mbox{ in }~~\Omega.
\end{equation*}
If the desired result has been established for the case $\xi=0$, then $v\in \mathcal S(1,p-1,0)$. Since $D^2 v= D^2 u$, it follows immediately that $u\in \mathcal S(1,p-1,0)$.

Henceforth, we assume that $\xi=0$. We only prove the supersolution property, since the proof of the subsolution property is analogous. More precisely, we show that if
\begin{equation*}
  \mathrm{div}(|Du|^{p-2}Du)\leq 0,
\end{equation*}
then
\begin{equation*}
  \mathcal M^{-}(D^2 u, 1, p-1)\leq 0.
\end{equation*}

Let $\varphi$ be a test function that touches $u$ strictly from below at a point $x\in \Omega$. For simplicity, we assume that $x=0$ and
\begin{equation*}
  \varphi(0)=u(0)=0 \ \ ~~\mbox{ and }~~\ \ \varphi<u \ \ ~~\mbox{ in }~~\ B_r\backslash \{0\}
\end{equation*}
for some $r>0$. Without loss of generality, we may assume that $\varphi$ is a quadratic polynomial of the form
\begin{equation*}
  \varphi=\frac{1}{2}Ax\cdot x+b\cdot x,
\end{equation*}
so that
\begin{equation*}
  D^2 \varphi =A\ \ ~~\mbox{ and }~~\ \ D \varphi (0)=b.
\end{equation*}
We shall prove that
\begin{equation*}
  \mathcal M^{-}(A, 1, p-1)\leq 0
\end{equation*}
by considering separately the cases $b\neq 0$ and $b=0$.

\textbf{Case (i) $b\neq 0$.}
By the definition of viscosity supersolutions, we have
\begin{equation*}
  \mathrm{div}(|D\varphi|^{p-2}D\varphi)
:=|D\varphi|^{p-2}\Delta\varphi
+(p-2)|D\varphi|^{p-4}D_i\varphi D_j\varphi D_{ij}\varphi
\leq 0
\end{equation*}
at the touching point. Therefore,
\begin{equation*}
  |b|^{p-2}\operatorname{tr} A +(p-2)|b|^{p-4}b_i  b_j A_{ij}
\leq 0.
\end{equation*}
Since $b\neq 0$, dividing both sides by $|b|^{p-2}$ yields
\begin{equation*}
  \operatorname{tr} A+(p-2)\frac{b_i  b_j}{|b|^2} A_{ij} \leq 0.
\end{equation*}
Observe that
\begin{equation*}
  \mathcal M^{-}(A, 1, p-1)\leq \operatorname{tr} A+(p-2)\frac{b_i  b_j}{|b|^2} A_{ij}.
\end{equation*}
Hence,
\begin{equation*}
  \mathcal M^{-}(A, 1, p-1)\leq 0.
\end{equation*}

\textbf{Case (ii) $b=0$.}
We argue by contradiction. Suppose that $\mathcal M^{-}(A, 1, p-1)> 0$.
Then $A$ has at least one positive eigenvalue. Let $S$ denote the direct sum of eigenspaces corresponding to the non-negative eigenvalues of $A$, and let $P_S$ be the orthogonal projection onto $S$. Define
\begin{equation*}
  \psi(x)=\varphi(x)+\varepsilon|P_S x|.
\end{equation*}
Since $\varphi<u$ in $B_r\backslash\{0\}$, it follows that, for sufficiently small $\varepsilon>0$, the function $u-\psi$ attains its minimum at some point $x_0\in B_r$.

We first claim that
\begin{equation*}
  P_S x_0\neq 0.
\end{equation*}
Indeed, suppose by contradiction that $P_S x_0= 0$. For any unit vector $\theta \in \mathbb{R}^n$, define
\begin{equation*}
  \psi_{\theta}(x)=\varphi(x)+\varepsilon \theta \cdot P_S x.
\end{equation*}
Since
\begin{equation*}
  |P_S x|=\max_{|\theta|=1}\theta \cdot P_S x,
\end{equation*}
we have
\begin{equation*}
  \psi_{\theta}(x)\leq \psi(x)\ \ ~~\mbox{ in }~~\ B_r.
\end{equation*}
Hence,
\begin{equation*}
  (u-\psi_{\theta})(x_0)=(u-\psi)(x_0)\leq (u-\psi)(x)\leq (u-\psi_{\theta})(x)\ \ ~~\mbox{ for all }~~\ x\in B_r,
\end{equation*}
which shows that $u-\psi_{\theta}$ also reaches its minimum at $x_0$.
Choose $\theta$ so that
\begin{equation*}
  D \psi_{\theta} (x_0)= Ax_0 +\varepsilon P_S \theta \neq 0.
\end{equation*}
By the definition of viscosity supersolutions,
\begin{equation*}
  \mathrm{div}(|D\psi_{\theta}|^{p-2}D\psi_{\theta})
=|D\psi_{\theta}|^{p-2}\Delta\psi_{\theta}
+(p-2)|D\psi_{\theta}|^{p-4}D_i\psi_{\theta} D_j\psi_{\theta} D_{ij}\psi_{\theta}
\leq 0
\end{equation*}
at $x_0$. Dividing both sides by $|D\psi_{\theta}(x_0)|^{p-2}$, we obtain
\begin{equation*}
  \Delta \psi_{\theta}(x_0)
+(p-2)\frac{D_i\psi_{\theta}(x_0) D_j\psi_{\theta}(x_0)}{|D\psi_{\theta}(x_0)|^2} D_{ij}\psi_{\theta}(x_0)
\leq 0.
\end{equation*}
Since
\begin{equation*}
  \mathcal M^{-}(D^2\psi_{\theta}(x_0), 1, p-1) \leq
\Delta \psi_{\theta}(x_0) +(p-2)\frac{D_i\psi_{\theta}(x_0) D_j\psi_{\theta}(x_0)}{|D\psi_{\theta}(x_0)|^2} D_{ij}\psi_{\theta}(x_0),
\end{equation*}
it follows that
\begin{equation*}
  \mathcal M^{-}(A, 1, p-1)=\mathcal M^{-}(D^2\psi_{\theta}(x_0), 1, p-1)\leq 0,
\end{equation*}
contradicting our assumption. Therefore,
\begin{equation*}
P_S x_0 \neq 0.
\end{equation*}

Next, we show that
\begin{equation*}
  D\psi(x_0)=Ax_0+\varepsilon \theta_0\neq 0,
\end{equation*}
where
\begin{equation*}
  \theta_0=\frac{P_S x_0}{|P_S x_0 |}.
\end{equation*}
Indeed, since $P_S x_0 \neq0$,
\begin{equation*}
  (Ax_0+\varepsilon \theta_0)\cdot P_S x_0= P_S Ax_0\cdot x_0+\varepsilon|P_S x_0|
\geq \varepsilon|P_S x_0|>0.
\end{equation*}
Hence,
\begin{equation*}
  D\psi(x_0)=Ax_0+\varepsilon \theta_0\neq 0.
\end{equation*}
Since $u-\psi$ attains its minimum at $x_0$, the viscosity inequality gives
\begin{equation*}
  |D\psi(x_0)|^{p-2}\Delta\psi(x_0)
+(p-2)|D\psi(x_0)|^{p-4}D_i\psi(x_0) D_j\psi(x_0)  D_{ij}\psi(x_0)
\leq 0.
\end{equation*}
Dividing both sides by $|D\psi(x_0)|^{p-2}$, we have
\begin{equation*}
  \Delta \psi (x_0)
+(p-2)\frac{D_i\psi(x_0) D_j\psi (x_0)}{|D\psi(x_0)|^2} D_{ij}\psi (x_0)
\leq 0.
\end{equation*}
Finally, note that the function $x\mapsto |P_Sx|$ is convex. Hence,
\begin{equation*}
  \mathcal M^{-}(A, 1, p-1)
\leq \mathcal M^{-}(D^2 \psi(x_0), 1, p-1)\leq 0,
\end{equation*}
which again contradicts our assumption.
\end{proof}

As a consequence of \Cref{l.vis.Pucci} and boundary $C^{1,\alpha}$ regularity for fully nonlinear equations (see \cite[Lemma 3.1]{MR3246039}), we have the following boundary $C^{1,\alpha}$ regularity for the $p$-Laplace equation in flat domains. 

\begin{theorem}\label{t.p-Lap.flat}
Let $\xi\in \mathbb{R}^n$. Assume that $u$ is a viscosity solution of
\begin{equation}
	\begin{cases}
    \mathrm{div} \left(|Du+\xi|^{p-2}(Du+\xi) \right)= 0   &~~\ \mbox{ in }~~\   B_r^+,\\
	u=0   &~~\ \mbox{ on }~~ \ T_r.
	\end{cases}
\end{equation}
Then there exists a constant $0<\bar \alpha<1$ depending on $n$ and $p$ such that $u\in C^{1,\bar\alpha}(0)$, i.e., there exists a constant $a_r$ such that
\begin{equation}\label{alpha1}
  |u(x)- a_rx_n|\leq   C_1\frac{|x|^{1+\bar \alpha}}{r^{1+\bar\alpha}}\|u\|_{L^\infty(B_r^+)},\ \ \forall x\in B_{r/2}^+
\end{equation}
and
\begin{equation}\label{alpha2}
  |Du(0)|=|a_r| \leq \frac{C_1}{r}\|u\|_{L^\infty(B_r^+)},
\end{equation}
where $C_1$ depends on $n$ and $p$.
\end{theorem}


The following lemma provides a comparison principle for the difference of two solutions to the $p$-Laplace equation. The proof is inspired by \cite[Lemma 2.2]{MR5035126}.

\begin{lemma}\label{ABPw}
Let $\xi \in \mathbb{R}^n$. Suppose that $u$ and $v$ are weak solutions of
\begin{equation*}
  \begin{cases}
	\mathrm{div}\left(|Du + \xi|^{p-2}(Du + \xi)\right) \geq f_1 &\ \mbox{in} \ \ \Omega, \\
		u \leq g_1 &\ \mbox{on}\ \ \partial\Omega,
	\end{cases}
\end{equation*}
and
\begin{equation*}
  \begin{cases}
		\mathrm{div}\left(|Dv + \xi|^{p-2}(Dv + \xi)\right) \leq f_2 &\ \mbox{in}\ \ \Omega, \\
		v \geq g_2 &\ \mbox{on}\ \ \partial\Omega,
	\end{cases}
\end{equation*}
respectively, where $f_1, f_2 \in L^q(\Omega)$ with $q > n/p$ and $1/p + 1/q \leq 1$. Then
\begin{equation}\label{abp}
	\sup_{\Omega}(u - v) \leq \|(g_1 - g_2)^+\|_{L^\infty(\partial\Omega)}
	+ C |\Omega|^{\frac{p}{n(p-1)} - \frac{1}{q(p-1)}} \|(f_1 - f_2)^-\|_{L^q(\Omega)}^{\frac{1}{p-1}},
\end{equation}
where $C$ depends only on $n$, $p$ and $q$.
\end{lemma}

\begin{proof}
For any $r \geq 0$, set
\begin{equation*}
  A(r) = \left\{ x \in \Omega : u(x) - v(x) - \|(g_1 - g_2)^+\|_{L^\infty(\partial\Omega)} \ge r \right\}.
\end{equation*}
Given $k \geq 0$, define
\begin{equation*}
  w(x) =
	\begin{cases}
		u(x) - v(x) - \|(g_1 - g_2)^+\|_{L^\infty(\partial\Omega)} - k, & x \in A(k), \\
		0, & x \in \Omega \setminus A(k).
	\end{cases}
\end{equation*}
By the boundary conditions,
\begin{equation*}
  u - v \leq g_1 - g_2 \leq (g_1 - g_2)^+~~\mbox{ on }~~\partial\Omega.
\end{equation*}
Hence,
\begin{equation*}
  w \geq 0~~\mbox{ in }~~\Omega, \quad   w = 0~~\mbox{ on }~~\partial\Omega.
\end{equation*}
By the standard monotonicity property of the $p$-Laplace operator,
\begin{equation*}
\begin{aligned}
  &\left(|Du + \xi|^{p-2}(Du + \xi) - |Dv + \xi|^{p-2}(Dv + \xi) \right) \cdot (Du - Dv)\\
  =&|Du + \xi|^p-|Du + \xi|^{p-2}(Du + \xi)\cdot (Dv + \xi)
  -|Dv + \xi|^{p-2}(Du + \xi)\cdot (Dv + \xi)+|Dv + \xi|^{p}\\
  =&\frac{1}{2}(|Du + \xi|^{p-2}+|Dv + \xi|^{p-2})|Du-Dv|^2
  +\frac{1}{2}(|Du + \xi|^{p-2}-|Dv + \xi|^{p-2})((|Du + \xi|^{2}-|Dv + \xi|^{2}))\\
  \geq& \frac{1}{2}(|Du + \xi|^{p-2}+|Dv + \xi|^{p-2})|Du-Dv|^2\\
  \geq& C |Du - Dv|^p,
\end{aligned}
\end{equation*}
where $C > 0$ depends only on $p$. Combining this estimate with the Poincar\'{e} inequality, we obtain
\begin{equation*}
  \begin{aligned}
		|A(k)|^{-p/n} \int_{A(k)} |w|^p
		& \leq C \int_{A(k)} |Dw|^p  = C \int_{A(k)} |Du - Dv|^p \\
		& \leq C \int_{A(k)} \left( |Du+\xi|^{p-2}(Du+\xi) - |Dv+\xi|^{p-2}(Dv+\xi)\right) \cdot (Du - Dv) dx \\
		&\leq C \int_{A(k)} (f_2 - f_1) w dx \\
		&\leq C \left( \int_{A(k)} w^p \right)^{1/p} \left( \int_{A(k)} |(f_1 - f_2)^-|^{p/(p-1)} \right)^{(p-1)/p}.
	\end{aligned}
\end{equation*}
Therefore,
\begin{equation*}
  \begin{aligned}
		\|w\|_{L^p(A(k))}^{p-1}
		& \leq C |A(k)|^{p/n} \|(f_1 - f_2)^-\|_{L^{p/(p-1)}(A(k))} \\
		& \leq C |A(k)|^{p/n + (p-1)/p - 1/q} \|(f_1 - f_2)^-\|_{L^q(A(k))}.
	\end{aligned}
\end{equation*}
For any $h \geq k$,
\begin{equation*}
  \begin{aligned}
		|A(h)|^{1/p} (h - k)
		& = \|h - k\|_{L^p(A(h))}
		\leq \|w\|_{L^p(A(h))}
		\leq \|w\|_{L^p(A(k))} \\
		& \leq C |A(k)|^{p/(n(p-1)) + 1/p - 1/(q(p-1))} \|(f_1 - f_2)^-\|_{L^q(A(k))}^{1/(p-1)}.
	\end{aligned}
\end{equation*}
Hence,
\begin{equation*}
  |A(h)| \leq \frac{C \|(f_1 - f_2)^-\|_{L^q(A(k))}^{p/(p-1)}}{(h - k)^p} |A(k)|^{p^2/(n(p-1)) + 1 - p/(q(p-1))}.
\end{equation*}

To conclude the proof, we apply the iteration lemma in \cite[Lemma 2.1]{MR5035126} with
\begin{equation*}
  \phi(r) = |A(r)|,\ \
  \beta = 1 + \frac{p^2}{n(p-1)} - \frac{p}{q(p-1)} > 1, \ \
  \gamma = p.
\end{equation*}
Since $\phi(0) \leq |\Omega|$, we obtain
\begin{equation*}
  \phi(r) = 0, \quad
  r = C 2^{\beta/(\beta-1)} \|(f_1 - f_2)^-\|_{L^q(\Omega)}^{1/(p-1)} |\Omega|^{p/(n(p-1)) - 1/(q(p-1))}.
\end{equation*}
Therefore, \eqref{abp} follows.
\end{proof}

We next present several preliminary results needed for the compactness argument. First, we need the following rough Lipschitz estimate.
\begin{lemma}\label{le2.1}
Let $\xi\in \mathbb{R}^n$ and $u$ be a viscosity solution of
\begin{equation}\label{e2.7}
	\mathrm{div}\left(|Du+\xi|^{p-2}(Du+\xi) \right)=f \ \ \mbox{in}\ B_1,
\end{equation}
satisfying
\begin{equation*}
  \operatorname{osc}_{B_1} u\leq 1,\quad \|f\|_{L^\infty(B_1)}\leq 1.
\end{equation*}
Then
\begin{equation}\label{e2.8}
  [u]_{C^{0,1}(\overline B_{1/2})}\leq C(1+|\xi|),
\end{equation}
where $C>0$ depends only on $n,p$.
\end{lemma}
\begin{proof}
Let $v=u-u(0)+\xi\cdot x$. Then $v$ is a solution of the classical $p$-Laplace equation:
\begin{equation*}
	\mathrm{div}\left(|Dv|^{p-2}Dv\right)=f \ \ \mbox{in}\ B_1.
\end{equation*}
By the interior Lipschitz regularity for $v$ (see \cite[Theorem 1]{MR709038}, \cite[Theorem 1]{MR727034}, \cite[Theorem 1.1]{MR2738325}, \cite[Theorem 1.1]{MR3800106}),
\begin{equation*}
[v]_{C^{0,1}(\overline B_{1/2})}\leq C\left(\|v\|_{L^{\infty}(B_1)}+\|f\|^{\frac{1}{p-1}}_{L^{\infty}(B_1)}\right)
\leq C\left(1+|\xi|\right),
\end{equation*}
where $C>0$ depends only on $n,p$. By transforming to $u$, we obtain \cref{e2.8}.
\end{proof}

Next, we establish the following interior $C^{\beta}$ estimates. Here both the H\"{o}lder exponent and the H\"{o}lder norm are uniform with respect to $\xi$.
\begin{lemma}\label{l3.3}
Let $u$ be a viscosity solution of
\begin{equation}\label{e}
	\mathrm{div}\left(|Du+\xi|^{p-2}(Du+\xi) \right)=f \ \ \mbox{in}\ B_1,
\end{equation}
satisfying
\begin{equation*}
  \operatorname{osc}_{B_1} u\leq 1,\quad \|f\|_{L^\infty(B_1)}\leq 1,
\end{equation*}

Then there exists a constant $0<\beta<1$, depending only on $n,p$, such that for every $0<r<1$,
\begin{equation}\label{e2.12}
  [u]_{C^\beta(\overline{B}_{r})}\leq C,
\end{equation}
where $C$ depends only on $n,p$ and $r$.
\end{lemma}

\begin{proof}
We only prove that
\begin{equation*}
[u]_{C^\beta(\overline{B}_{1/2})}\leq C,
\end{equation*}
where $C$ depends only on $n,p$. Then a standard covering argument leads to \cref{e2.12}.

We distinguish two cases. Let $a_0=a_0(n,p)$ be a large constant to be specified later.

\textbf{Case (i): $|\xi|\geq a_0$.}
To prove the H\"{o}lder  continuity, we use the Ishii-Lions argument (see \cite{MR1031377,MR2995669,MR3800106,LLYZ2025}). For any $x_0 \in B_{1/2}$, consider
\begin{equation}\label{key}
  M=\sup_{x,y\in \overline{B}_{3/4}}
  \left(u(x) - u(y) - L_1|x - y|^{\gamma} - L_2|x - x_0|^2 - L_2|y - x_0|^2\right),
\end{equation}
where $L_1\gg L_2$ are two large constants and $\gamma>0$ is a small constant to be specified later.

If we can prove that $M\leq 0$, then for any $y\in \overline B_{1/2}$,
\begin{equation*}
u(x_0) - u(y) - L_1|x_0 - y|^{\gamma}  - L_2|y - x_0|^2\leq M\leq 0
\end{equation*}
and
\begin{equation*}
u(y) - u(x_0) - L_1|y - x_0|^{\gamma}  - L_2|y - x_0|^2\leq M\leq 0,
\end{equation*}
which implies that $u\in C^{\gamma}(x_0)$. Since $x_0\in B_{1/2}$ is arbitrary, we obtain $u\in C^{\gamma}(\overline{B}_{1/2})$.

In the following, we prove that $M\leq 0$ by contradiction. Assume that $M>0$. Let $(x,y)$ be a maximum point in $\overline{B}_{3/4}\times \overline{B}_{3/4}$. Then
\begin{equation}\label{e2.9}
  L_1|x-y|^{\gamma} + L_2|x - x_0|^2 + L_2|y - x_0|^2
  \leq |u(x)-u(y)| \leq 1.
\end{equation}
Choosing $L_2> 16$ gives
\begin{equation}\label{x0y0}
  |x-x_0|,|y-x_0|< \frac{1}{4}.
\end{equation}
Hence $x,y\in B_{3/4}$. Moreover, $x\neq y$, otherwise $M\leq 0$.
Set
\begin{equation*}
  \delta=x-y,\quad \hat{\delta}=\frac{\delta}{|\delta|},\quad
  \hat q=L_1\gamma|\delta|^{\gamma-1}\hat{\delta},
\end{equation*}
and define
\begin{equation}\label{qxqy}
  q_x =\hat q + 2L_2(x - x_0) \quad \text{and} \quad q_y = \hat q - 2L_2(y - x_0).
\end{equation}
Moreover, let
\begin{equation*}
  Z = L_1 \left(\gamma(\gamma-1)|\delta|^{\gamma-2}\hat{\delta}\otimes\hat{\delta}
  +\gamma|\delta|^{\gamma-2}(I-\hat{\delta}\otimes\hat{\delta})\right)
  =L_1\left(\gamma(\gamma-2)|\delta|^{\gamma-2}\hat{\delta}\otimes\hat{\delta}
  +\gamma|\delta|^{\gamma-2}I\right).
\end{equation*}
and
\begin{equation*}
  A=
  \begin{pmatrix}
		Z & -Z \\
		-Z & Z
	\end{pmatrix}
	+ 2L_2I,
\end{equation*}
where $I$ denotes the $2n\times 2n$ identity matrix.

To obtain appropriate viscosity inequalities, we apply the standard Jensen-Ishii’s Lemma (see \cite[Theorem 3.2]{MR1118699}, also \cite{MR1031377, MR2911421}) with a parameter $\varepsilon>0$. By choosing $\varepsilon$ sufficiently small (depending on $\|Z\|$) such that the higher-order term $\varepsilon A^2$ is bounded by $I$, we can construct a limiting super-jet $(q_x,X)$ of $u$ at $x$ and a limiting sub-jet $(q_y, Y)$ of $u$ at $y$. Then, the matrices $X$ and $Y$ satisfy the following $2n\times 2n$ matrix inequality:
\begin{equation}\label{XYMatrix}
  \begin{pmatrix}
		X & 0 \\
		0 & -Y
	\end{pmatrix}
	\leq
	\begin{pmatrix}
		Z & -Z \\
		-Z & Z
	\end{pmatrix}
	+ (2L_2+1)I.
\end{equation}
Applying \eqref{XYMatrix} to $(v,v)\in \mathbb{R}^n\times \mathbb{R}^n$, we have
\begin{equation*}
  \langle (X-Y)v,v\rangle \leq (4L_2+2)|v|^2\leq 5L_2|v|^2.
\end{equation*}
Thus, $X-Y\leq 5L_2I$, i.e., all eigenvalues of $X-Y$ are at most $5L_2$. In addition, applying the particular vector $(\hat{\delta},-\hat{\delta})$, we obtain
\begin{equation*}
  \begin{aligned}
  \langle (X-Y)\hat{\delta},\hat{\delta}\rangle
  \leq 4\langle Z\hat{\delta},\hat{\delta}\rangle+ (4L_2+2)
  \leq 4L_1\gamma(\gamma-1)|\delta|^{\gamma-2}+ 5L_2.
  \end{aligned}
\end{equation*}
Since $0<\gamma<1$, at least one eigenvalue of $X-Y$ is at most $4L_1\gamma(\gamma-1)|\delta|^{\gamma-2}+ 5L_2$ which will be negative if $L_1\gg L_2$ is large enough. As in \cite[Proof of Lemma 4]{MR2995669}, this yields
\begin{equation}\label{Part1}
  \mathcal{M}^+(X - Y,1,p-1) \leq  CL_2+4L_1\gamma(\gamma-1)|\delta|^{\gamma-2},
\end{equation}
where $C$ depends only on $n,p$.

By the interior Lipschitz regularity for $u$ (see \Cref{le2.1}) and choosing $a_0\geq 1$, we have
\begin{equation*}
|u(x)-u(y)|\leq C(1+|\xi|)|x-y|\leq C_0|\xi||x-y|,
\end{equation*}
where $C_0$ depends only on $n$ and $p$. Combining with \cref{e2.9} gives
\begin{equation}\label{e2.10}
L_1|x-y|^{\gamma}\leq |u(x)-u(y)|\leq C_0|\xi||x-y|.
\end{equation}
We further choose $a_0\geq 16L_2$ and $\gamma=  1/(4C_0)$. Then
\begin{equation*}
|q_x|= \left|L_1\gamma|\delta|^{\gamma-1}\hat{\delta}+2L_2(x - x_0)\right|\leq C_0\gamma |\xi|+4L_2\leq |\xi|/2.
\end{equation*}
Similarly, $|q_y|\leq |\xi|/2$. Hence,
\begin{equation*}
\frac{|\xi|}{2}\leq  |q_x+\xi|,|q_y+\xi|\leq \frac{3|\xi|}{2}.
\end{equation*}
Additionally, the viscosity inequalities
\begin{equation*}
  |q_x+\xi|^{p-2}F(q_x, X)\geq f(x), \quad
 |q_y+\xi|^{p-2}F(q_y, Y) \leq f(y),
\end{equation*}
where
\begin{equation*}
  F(q,X)=\mathrm{tr} X + (p-2)\frac{(q_i+\xi_i)(q_j+\xi_j)}{|q+\xi|^{2}}
	X_{ij},
\end{equation*}
 yield
\begin{equation*}
  F(q_x,X)\geq -(a_0/2)^{2-p}\geq -1 ~~~~\mbox{ and }~~~~
  F(q_y,Y)\leq (a_0/2)^{2-p}\leq 1.
\end{equation*}
Then
\begin{equation}\label{Fqxqy}
  F(q_x,X)-F(q_y,Y)\geq -2.
\end{equation}

Let
\begin{equation*}
  P(v)=\frac{v \otimes v}{|v|^2}.
\end{equation*}
By noting \eqref{x0y0} and \eqref{qxqy}, we have
\begin{equation*}
  |q_x-q_y|\leq 2L_2\left(|x-x_0|+|y-x_0|\right)\leq 2L_2.
\end{equation*}
Then by a direct calculation, we have
\begin{equation}\label{e2.11}
  |P(q_x+\xi)-P(q_y+\xi)|  \leq \frac{4|q_x-q_y| }{\min(|q_x+\xi|,|q_y+\xi|)}\leq 8 |\xi|^{-1}|q_x-q_y|  \leq 16L_2a_0^{-1}.
\end{equation}

Next, by applying \eqref{XYMatrix} to $(v,0)\in \mathbb{R}^n\times \mathbb{R}^n$, we obtain
\begin{equation*}
  \langle Xv,v\rangle \leq \langle Zv,v\rangle+(2L_2+1)|v|^2\leq \langle Zv,v\rangle+3L_2|v|^2.
\end{equation*}
Thus, $X\leq Z+3L_2I$, i.e., all eigenvalues of $X$ are at most $\|Z\|+3L_2$. In addition, if the minimum eigenvalue of $X$ is negative, denoted by $\mu$, we have
\begin{equation*}
-1\leq F(q_x,X)  \leq \mathcal{M}^+(X,1,p-1)
  \leq (n-1)(p-1)\left(\|Z\|+3L_2\right)+\mu.
\end{equation*}
Then,
\begin{equation*}
  \mu\geq -(n-1)(p-1)\left(\|Z\|+3L_2\right)-1.
\end{equation*}
That is, the eigenvalues of $X$ have a lower bound. Hence,
\begin{equation*}
  \|X\|\leq C \left(\|Z\|+L_2\right),
\end{equation*}
where $C$ depends only on $n,p$.

By noting that
\begin{equation*}
  \|Z\|\leq L_1\gamma |\delta|^{\gamma-2},
\end{equation*}
we obtain
\begin{equation}\label{Part3}
\begin{aligned}
|F(q_y,X)-F(q_x,X)|
\leq& C|P(q_x+\xi)-P(q_y+\xi)| \|X\|\\
  \leq& CL_2a_0^{-1}\left(L_1\gamma|\delta|^{\gamma-2}+L_2\right),
\end{aligned}
\end{equation}
where $C$ depends only on $n$ and $p$.
Note that
\begin{equation}\label{Part2}
  \begin{aligned}
   F(q_x,X)-F(q_y,Y)
   =& F(q_x,X)-F(q_y,X)+ F(q_y,X)-F(q_y,Y)\\
   \leq & |F(q_x,X)-F(q_y,X)|+\mathcal{M}^+(X - Y,1,p-1).
  \end{aligned}
\end{equation}
By combining with \cref{Part1}, \cref{Fqxqy} and \cref{Part3}, we arrive at
\begin{equation*}
-2\leq C_1a_0^{-1}L_1\gamma|\delta|^{\gamma-2}+C_2+4L_1\gamma(\gamma-1)|\delta|^{\gamma-2},
\end{equation*}
where $C_1$ and $C_2$ depend only on $n,p,L_2$. Finally, choose $a_0$ large enough such that
\begin{equation*}
C_1a_0^{-1}\leq 1-\gamma
\end{equation*}
and then choose $L_1$ large enough such that
\begin{equation*}
L_1\gamma(1-\gamma)>C_2+2.
\end{equation*}
Since $0<|\delta|<1/2$ and $0<\gamma<1$, we have $|\delta|^{\gamma-2}\ge1$. Consequently,
\begin{equation*}
  -2\leq C_2-\bigl[4(1-\gamma)-C_1a_0^{-1}\bigr]L_1\gamma|\delta|^{\gamma-2}
\leq C_2-3L_1\gamma(1-\gamma)|\delta|^{\gamma-2}
\leq C_2-L_1\gamma(1-\gamma)<-2,
\end{equation*}
which is a contradiction.
Therefore,
\begin{equation*}
  [u]_{C^{\gamma}(\overline B_{1/2})}\leq C,
\end{equation*}
where $C$ depends only on $n,p$.	

\medskip

\textbf{Case (ii). The case $|\xi|\leq a_0$.}
Since
\begin{equation*}
  \mathcal M^{-}(X,1,p-1)\leq \mathrm{tr} X
+(p-2)\frac{(q_i+\xi_i)(q_j+\xi_j)}{|q+\xi|^{2}} X_{ij}
\leq \mathcal M^{+}(X,1,p-1).
\end{equation*}
and
\begin{equation*}
|q+\xi|^{p-2}\geq a_0^{p-2} ~~\mbox{ whenever }~~|q|\geq 2a_0
\end{equation*}
every solution of \eqref{e} satisfies
\begin{equation*}
  \begin{cases}
    \mathcal M^{+}(D^2 u,1,p-1) + {a}_0^{2-p} |f| \geq 0,& \\
    \mathcal M^{-}(D^2 u,1,p-1)-  {a}_0^{2-p} |f| \leq 0,&
  \end{cases}
\end{equation*}
whenever $|Du|\geq 2a_0$. Therefore, from \cite[Theorem 1.1]{MR3500837}, there exists $\tau\in (0,1)$, depending only on $n,p$, such that
\begin{equation*}
  [u]_{C^{\tau}(\overline B_{1/2})}\leq C\left(\operatorname{osc}_{B_1} u+ \max \left\{2a_0,\|f\|_{L^n(B_1)}\right\}\right)
  \leq C,
\end{equation*}
where $C$ depends only on $n,p$. Finally, by setting $\beta=\min(\gamma,\tau)=\min(1/(4C_0),\tau)$, the proof is completed.
\end{proof}

Next, we establish the following uniform boundary estimate.

\begin{lemma}\label{d1}
Let $\xi\in \mathbb{R}^n$ and $0<\delta<1/4$. Suppose that $u$ is a viscosity solution of
\begin{equation*}
	\begin{cases}
      \mathrm{div}\left(|Du+\xi|^{p-2}(Du+\xi) \right)=f  &~~\mbox{ in }~~ \Omega_1,\\
	  u=g   &~~\mbox{on}~~(\partial\Omega)_1,
	\end{cases}
\end{equation*}
where $f\in L^q(\Omega_1)$ with $q>n/p$ and $1/p+1/q\leq 1$.
If
$\|u\|_{L^\infty(\Omega_1)}\leq 1$ and
\begin{equation*}
  \|f\|_{L^q(\Omega_1)}\leq\delta^{p-1},\
  \|g\|_{L^\infty((\partial\Omega)_1)}\leq\delta,\
  \operatorname{osc}_{B_1}\partial\Omega \leq\delta,
\end{equation*}
where $\operatorname{osc}_{B_1}\partial\Omega \leq\delta$ means
\begin{equation*}
  B_1\cap \{x_n>\delta\}\subset \Omega_1 \subset B_1\cap \{x_n >-\delta\},
\end{equation*}
then
\begin{equation*}
  |u(x)|\leq C(x_n+2\delta),\quad x\in \Omega_{1/2},
\end{equation*}
where $C$ depends only on $n,p$ and $q$.
\end{lemma}

\begin{proof}
If $\delta\geq1/8$, then
\begin{equation*}
  |u(x)|\leq1\leq8(x_n+2\delta),\qquad x\in\Omega_{1/2}.
\end{equation*}
Hence, we may assume that $\delta<1/8$.
By a rotation, we may assume that $\Omega_1\subset B_1\cap \{x_n>-\delta\}$.
Set $\widetilde B_r^+=B_r^+-\delta e_n$ and
$\widetilde T_r=T_r-\delta e_n$. Then
\begin{equation*}
  \Omega_{1/2}\subset\widetilde B_{5/8}^+
\subset\widetilde B_{3/4}^+\subset B_1.
\end{equation*}

Let $v$ be a solution of
\begin{equation*}
  \begin{cases}
    \mathrm{div}\left(|Dv+\xi|^{p-2}(Dv+\xi)\right)= 0 & ~~\mbox{ in }~~ \widetilde{B}_{3/4}^+, \\
    v= 0 & ~~\mbox{ on }~~ \widetilde{T}_{3/4}, \\
    v= 1 & ~~\mbox{ on }~~ \partial \widetilde{B}_{3/4}^+ \backslash \widetilde{T}_{3/4}.
  \end{cases}
\end{equation*}
By \Cref{l.vis.Pucci}, $v \in \mathcal S(1,p-1,0)$.
By the boundary pointwise $C^{1,\alpha}$ regularity of the Pucci class (\cite[Lemma 3.1]{MR3246039}),
\begin{equation*}
  |v(x)|\leq C(x_n+\delta) ~~\mbox{ in }~~ \widetilde{B}_{5/8}^+.
\end{equation*}
In addition, applying \Cref{ABPw} in $\Omega\cap\widetilde B_{3/4}^+\subset\Omega_1$, we have
\begin{equation*}
  \sup_{\Omega\cap  \widetilde{B}^+_{3/4}}(u-v)\leq C \delta,
\end{equation*}
where $C$ depends only on $n, p$ and $q$.
Then,
\begin{equation*}
  u(x)
  \leq v(x)+C \delta
  \leq C(x_n+2\delta).
\end{equation*}
The proof for
\begin{equation*}
  u(x)\geq -C(x_n+2\delta)
\end{equation*}
is similar and we omit it. Hence, the proof is completed.

\end{proof}

\section{Boundary pointwise Lipschitz regularity}
\label{sec:Lip}
In this section, we prove the boundary pointwise Lipschitz regularity stated in \Cref{t.Lip}. The proof follows the iteration scheme developed in \cite{MR4713521, DLL_2025}. 


\begin{proof}
We first make some normalization. Let
\begin{equation*}
	M=\|u\|_{L^{\infty}(\Omega_1)}
    +\|f\|_{C_q^{-1,\text{Dini}}(0)}^{1/(p-1)}+\|g\|_{C^{1,\mathrm{Dini}}(0)}
\end{equation*}
and
\begin{equation*}
  \tilde{u}=\frac{u-g(0)-Dg(0)\cdot x}{M}.
\end{equation*}
Then $\tilde{u}$ satisfies
\begin{equation}\label{e.p-lap.tilde}
\begin{cases}
	\mathrm{div}\left(|D \tilde u+Dg(0)/M|^{p-2} (D \tilde u+Dg(0)/M)\right) =\tilde f & \text{in } \Omega_1, \\
	\tilde u = \tilde g & \text{on } (\partial\Omega)_1,
\end{cases}
\end{equation}
where
\begin{equation*}
  \tilde f=\frac{f}{M^{p-1}}\ \ \mathrm{and}\ \ \tilde g=\frac{g -g(0)-Dg(0)\cdot x}{M}.
\end{equation*}
Hence,
\begin{equation*}
 \tilde u(0)=\tilde g(0)=|D \tilde g(0)|=0.
\end{equation*}
In the following, we only need to prove the boundary pointwise Lipschitz regularity for $\tilde u$.
Define
\begin{equation}\label{omg_Lip}
	\omega(r)=\max\{\omega_\Omega(r),\omega_{\tilde f}^{1/(p-1)}(r),\omega_{\tilde g}(r)\}.
\end{equation}
Without loss of generality, we may assume that
\begin{equation*}
  \omega(1)\leq \eta^2,\quad \int_0^{1}\frac{\omega(r)}{r}dr\leq \eta^2,
\end{equation*}
where $0<\eta<1/8$ is a constant depending only on $n$ and $p$, to be specified later.

We now use an iteration argument. We establish the following inductive claim:
there exist positive constants $\bar C$, $\widetilde C$ and a nonnegative sequence $a_k$ ($k\geq -1$) such that for every $k=0,1,2,...$,
\begin{equation}\label{sup-Lip}
	\sup_{\Omega_{\eta^k}}(\tilde{u}-a_k x_n)\leq \widetilde C   \eta^k A_k
\end{equation}
and
\begin{equation}\label{ak-Lip}
	|a_k-a_{k-1}|\leq \bar C \widetilde C  A_k
\end{equation}
with
\begin{equation}\label{A-Lip}
	A_0=\eta^2,\quad A_{k+1}=\max\{\omega(\eta^k),\eta^{\bar \alpha/2}A_{k}\},
\end{equation}
where $0<\bar \alpha<1$ is the constant from \Cref{t.p-Lap.flat}, and $\bar C$, $\widetilde C$ depend only on $n,p$ and $q$.

\medskip

For $k=0$, set $a_{-1}=a_0=0$. Choosing $\widetilde C$ sufficiently large so that
\begin{equation}\label{3.11}
	\widetilde C A_0=\widetilde C \eta^2 \geq 1,
\end{equation}
immediately yields \eqref{sup-Lip} and \eqref{ak-Lip} for  $k=0$.
Assume that \eqref{sup-Lip} and \eqref{ak-Lip} hold for some $k\geq 0$. We prove that they also hold for $k+1$.

Set
\begin{equation*}
  r=\eta^k/2,\ \ \Omega_r=\Omega\cap B_r, \ \ (\partial\Omega)_r=\partial\Omega\cap B_r,\ \ \widetilde \Omega_r=\Omega\cap \widetilde B_r^+,
\end{equation*}
where
\begin{equation*}
  \widetilde B_r^+:=B_r^+(-r\omega(r)e_n)=B_r^+-r\omega(r)e_n, \ \ \
\widetilde T_r:=T_r(-r\omega(r)e_n)=T_r-r\omega(r)e_n.
\end{equation*}

We first observe that
\begin{equation*}
  \Omega_{r/2} \subset \widetilde{\Omega}_r \subset \Omega_{2r},\ \ \
  \Omega_{\eta^{k+1}} = \Omega_{2\eta r}  \subset \widetilde{B}_{r/2}^+.
\end{equation*}
Indeed, since
\begin{equation*}
  \omega(\eta)\leq\omega(1)\leq \eta^2\leq \frac{1}{8},
\end{equation*}
the first inclusion is immediate. For the second one, if $x\in \Omega_{\eta^{k+1}}$, then
\begin{equation*}
  |x+r\omega(r)e_n|<2\eta r+\eta^2 r<\frac{r}{2},
\end{equation*}
while the exterior $C^{1,\mathrm{Dini}}$ condition at $0$ gives $x_n>-r\omega(r)$. Hence, $x \in \widetilde B_{r/2}^+$.

We now prove the inductive step by a perturbation argument. Following the idea of \cite{MR4713521, DLL_2025}, we regard $\tilde f$, $\tilde g$ and $\partial\Omega$ as perturbations of $0$, $0$ and a hyperplane, respectively.

Set $\xi_0=Dg(0)/M$. 
Let $v$ be a viscosity solution of
\begin{equation}\label{v}
	\begin{cases}
	\mathrm{div} \left(|Dv+a_k e_n+\xi_0|^{p-2}(Dv+a_k e_n+\xi_0) \right)= 0  &~~\mbox{ in }~~ \ \   \widetilde B_r^+,\\
    v=0\   &~~\mbox{ on }~~   \ \ \widetilde T_r,\\
    v= \widetilde C  \eta^k A_k     &~~\mbox{ on }~~  \ \   \partial \widetilde B_r^+\backslash \widetilde T_r.
	\end{cases}
\end{equation}
By \Cref{t.p-Lap.flat}, there exists a constant $a_r$ such that
\begin{equation}\label{v3'}
	\begin{aligned}
		\|v-a_r (x_n+r\omega(r))\|_{L^{\infty}(\Omega_{\eta^{k+1}})}
		&\leq \frac{C_1}{r^{1+\bar \alpha}}(2\eta r+\eta^2 r)^{1+\bar \alpha}	\|v\|_{L^{\infty}(\widetilde B_r^+)}\\
		&\leq \frac{C_1(3\eta r)^{1+\bar \alpha}}{r^{1+\bar \alpha}}
		\widetilde C   \eta^{k} A_{k}
		\leq 9 C_1 \eta^{\frac{\bar \alpha}{2}}
		\widetilde C   \eta^{k+1} A_{k+1}.
	\end{aligned}
\end{equation}
and
\begin{equation}\label{ar}
	\begin{aligned}
		|a_r|
		\leq \frac{C_1}{r}	\|v\|_{L^{\infty}(\widetilde B_r^+)}
		\leq \frac{C_1 \widetilde C  }{r}	\eta^{k} A_{k},
	\end{aligned}
\end{equation}
where $C_1$ and $\bar \alpha$ are given by \Cref{t.p-Lap.flat}.
Since $\omega(r)\leq \eta^2$,
\begin{equation}\label{v4'}
		|a_r r\omega(r)|
		\leq \frac{C_1 \widetilde C }{r}	\eta^{k} A_{k} r \eta^2
		\leq C_1  \eta^{1-\frac{\bar \alpha}{2}} \widetilde C \eta^{k+1} A_{k+1}.
\end{equation}
Consequently,
\begin{equation}\label{v2}
	\begin{aligned}
		\|v-a_r x_n\|_{L^\infty (\Omega_{\eta^{k+1}})}
		\leq \left(9C_1 \eta^{\frac{\bar \alpha}{2}}+C_1 \eta^{1-\frac{\bar \alpha}{2}}\right)
        \widetilde C  \eta^{k+1} A_{k+1}.
	\end{aligned}
\end{equation}

\medskip

Set
\begin{equation*}
  \bar{u}=\tilde u-a_k x_n.
\end{equation*}
Then $\bar{u}$ satisfies
\begin{equation}\label{tilde_u}
	\begin{cases}
	\mathrm{div} (|D\bar{u}+a_k e_n+\xi_0|^{p-2}(D\bar{u}+a_k e_n+\xi_0))=\tilde f  &~~\mbox{ in }~~  \Omega\cap \widetilde B_r^+,\\
	\bar{u}=\tilde g-a_k x_n   &~~\mbox{ on }~~     \partial\Omega\cap \widetilde B_r^+,\\
    \bar{u}\leq \widetilde C   \eta^k A_k    &~~\mbox{ on }~~     \Omega\cap \partial \widetilde B_r^+.
	\end{cases}
\end{equation}
Applying \Cref{ABPw} yields
\begin{equation*}
	\begin{aligned}
	\sup_{\Omega_{\eta^{k+1}}} (\bar{u}-v)
	\leq \sup_{\widetilde \Omega_{r}} (\bar{u}-v)
	&\leq \|(\tilde g-a_k x_n)^+\|_{L^\infty(\partial\Omega\cap\widetilde B_r^+)}+C_2 r^{\frac{p}{p-1}-\frac{n}{q(p-1)}}\|\tilde f\|^{\frac{1}{p-1}}_{L^{q}(\widetilde \Omega_{r})}\\
	&\leq \|\tilde g\|_{L^\infty(\partial\Omega\cap\widetilde B_r^+)}+a_k r\omega(r)+C_2 r^{\frac{p}{p-1}-\frac{n}{q(p-1)}}\|\tilde f\|^{\frac{1}{p-1}}_{L^{q}(\widetilde \Omega_{r})},
	\end{aligned}
\end{equation*}
where $C_2$ depends on $n,p$ and $q$.

Since $\tilde g\in C^{1,\mathrm{Dini}}(0)$ with $\tilde g(0)=|D \tilde g(0)|=0$, we have
\begin{equation}\label{J1}
	\begin{aligned}
	    \|\tilde g\|_{L^\infty(\partial\Omega\cap\widetilde B_r^+)}
		&\leq  \|\tilde g\|_{L^\infty (\partial\Omega\cap B_{2r})}
		\leq   2r \omega(2r)\leq   \eta^k \omega(\eta^k)\\
		&=\frac{1}{\eta^{1+\frac{\bar \alpha}{2}}}  \eta^{k+1} \eta^{\frac{\bar \alpha}{2}}\omega(\eta^k)
		\leq \frac{1}{\eta^{1+\frac{\bar \alpha}{2}}}  \eta^{k+1} A_{k+1}.
	\end{aligned}
\end{equation}

In addition, by the definition of $A_k$ (\eqref{A-Lip}) and the Dini condition on $\omega$, and the argument in \cite[(3.15) on p. 11]{DLL_2025}, we obtain
\begin{equation}\label{Ak}
	\sum\limits_{k=0}^{\infty} A_k\leq 4\eta^2,
\end{equation}	
after choosing $\eta$ sufficiently small.

Therefore, using the inductive estimate for $a_i-a_{i-1}$ (\eqref{ak-Lip}), we obtain
\begin{equation}\label{J2}
	\begin{aligned}
    |a_k| r\omega(r)
	&\leq \eta^{k} \omega(\eta^k)\sum_{i=0}^{k}|a_i-a_{i-1}|
	\leq \eta^{k} \omega(\eta^k)\bar C\widetilde C   \sum_{i=0}^{k} A_i\\
	&\leq 4\eta^2 \eta^{k} \omega(\eta^k)\bar C\widetilde C
	\leq 4 \eta^{1-\frac{\bar \alpha}{2}} \bar C \widetilde C   \eta^{k+1}A_{k+1}.
\end{aligned}
\end{equation}
Since $\tilde f\in C_q^{-1,\mathrm{Dini}}(0)$,
\begin{equation*}
  \begin{aligned}
	\| {\tilde f}\|^{\frac{1}{p-1}}_{L^{q}(\widetilde \Omega_{r})}
	\leq \| {\tilde f}\|^{\frac{1}{p-1}}_{L^{q}(\Omega_{2r})}
    \leq    (2r)^{\frac{-p}{p-1} + \frac{n}{q(p-1)}} \left(2r\omega_{ {\tilde f}}^{1/(p-1)}(2r) \right)
	\leq    (2r)^{\frac{-p}{p-1} + \frac{n}{q(p-1)}} \left(2r\omega(2r) \right).
\end{aligned}
\end{equation*}
Consequently,
\begin{equation}\label{J3}
	\begin{aligned}
   C_2 r^{\frac{p}{p-1}-\frac{n}{q(p-1)}}\|\tilde f\|^{\frac{1}{p-1}}_{L^{q}(\widetilde \Omega_{r})}
   &\leq C_2 r^{\frac{p}{p-1}-\frac{n}{q(p-1)}}\cdot
    (2r)^{\frac{-p}{p-1} + \frac{n}{q(p-1)}}(2r\omega(2r))\\
	&	\leq C_3  \eta^{k}\omega(\eta^k)
		\leq \frac{C_3}{\eta^{1+\frac{\bar \alpha}{2}}} \eta^{k+1} A_{k+1},
	\end{aligned}
\end{equation}
where $C_3$ depends on $n,p$ and $q$.
Combining \eqref{J1}, \eqref{J2} and \eqref{J3}, we obtain that
\begin{equation}\label{w2}
	\begin{aligned}
		\sup_{\Omega_{\eta^{k+1}}} (\bar u-v)
		&\leq \|\tilde g\|_{L^\infty(\partial\Omega\cap\widetilde B_r^+)}+a_k r\omega(r)+C_2 r^{\frac{p}{p-1}-\frac{n}{q(p-1)}}\|\tilde f\|^{\frac{1}{p-1}}_{L^{q}(\widetilde \Omega_{r})}\\
		&\leq \frac{1}{\eta^{1+\frac{\bar \alpha}{2}}}  \eta^{k+1} A_{k+1}
		+4 \eta^{1-\frac{\bar \alpha}{2}} \bar C \widetilde C    \eta^{k+1}A_{k+1}
		+\frac{C_3}{\eta^{1+\frac{\bar \alpha}{2}}} \eta^{k+1} A_{k+1}
		\\
		& \leq \left(\frac{C_3+1}{\widetilde{C}\eta^{1+\frac{\bar \alpha}{2}}}+
		4 \bar C \eta^{1-\frac{\bar \alpha}{2}} \right)
		\widetilde{C} \eta^{k+1} A_{k+1}.	
	\end{aligned}
\end{equation}
By taking $\eta$ small enough, we have
\begin{equation*}
	9 C_1  \eta^{\frac{\bar \alpha}{2}}<\frac{1}{3}, \ 9 C_1  \eta^{1-\bar \alpha}<\frac{1}{3}
\end{equation*}
and then choosing $\widetilde{C}$ large enough, we get
\begin{equation*}
	\frac{C_3+1}{\widetilde{C}\eta^{1+\frac{\bar \alpha}{2}}}<\frac{1}{3}.
\end{equation*}

Let
\begin{equation*}
  a_{k+1}=a_k+a_r.
\end{equation*}
By \eqref{ar} and choosing
\begin{equation*}
	\bar C=\frac{2 C_1}{\eta^{\bar\alpha/2}},
\end{equation*}
we have
\begin{equation*}
	|a_{k+1}-a_k|=|a_r|
	\leq \frac{C_1}{r}	\widetilde C   \eta^{k} A_{k}
	\leq \frac{2 C_1}{\eta^{\bar \alpha/2}}\widetilde C   A_{k+1}
	=\bar{C} \widetilde C   A_{k+1}.
\end{equation*}
Finally, combining \eqref{v2} and \eqref{w2}, we conclude that
\begin{equation*}
	\begin{aligned}
		\sup_{\Omega_{\eta^{k+1}}}(\tilde u-a_{k+1} x_n)
		=&\sup_{\Omega_{\eta^{k+1}}}(\bar u-v+v-a_r x_n)\\
		\leq& \left(\frac{C_3+1}{\widetilde{C}\eta^{1+\frac{\bar \alpha}{2}}}
		+4 \bar{C} \eta^{1-\frac{\bar \alpha}{2}}
		+9C_1 \eta^{\frac{\bar \alpha}{2}}+ C_1  \eta^{1-\frac{\bar \alpha}{2}}\right)\widetilde C   \eta^{k+1} A_{k+1}\\
		\leq &\left(\frac{C_3+1}{\widetilde{C}\eta^{1+\frac{\bar \alpha}{2}}}
		+9 C_1 \eta^{1-\bar \alpha}
		+9C_1 \eta^{\frac{\bar \alpha}{2}}\right)\widetilde C   \eta^{k+1} A_{k+1}\\
		\leq& \widetilde C   \eta^{k+1} A_{k+1}.
	\end{aligned}
\end{equation*}
Thus, both \eqref{sup-Lip} and \eqref{ak-Lip} hold for $k+1$. The claim then follows by induction.

The regularity $u\in C^{0,1}(0)$ then follows from a standard argument (see \cite[Step 3 in Section 3.2]{DLL_2025}).


\end{proof}

\section{The Hopf lemma}
\label{sec:Hopf}
In this section, we establish the Hopf lemma (\Cref{t.Hopf}). We first prove the \emph{Hopf lemma on flat boundaries} by constructing a barrier function and applying the comparison principle.

\begin{lemma}\label{hpflat}
Let $q>n/p$ and $1/p + 1/q \leq 1$, and let $u\geq 0$ be a viscosity solution of
\begin{equation}\label{v-hopf-flat}
	\begin{cases}
      \mathrm{div}(|Du|^{p-2}Du)=f \ \ &\mbox{in} \ \  B_1^+,\\
	  u \geq 0 \ \  &\mbox{on}   \ \  T_1,
	\end{cases}
\end{equation}	
satisfying
\begin{equation*}
  u(e_n/2)\geq 1.
\end{equation*}
Then there exist constants $\delta_0,c,C>0$, depending only on $n,p$ and $q$, with $\delta_0$ sufficiently small, such that if
\begin{equation*}
  \|f\|_{L^q(B_1^+)}\leq \delta_0,
\end{equation*}
then
\begin{equation*}
  u(x)\geq c x_n - C\|f\|_{L^q(B_1^+)}^{1/(p-1)}
\end{equation*}
for every $x\in B_{1/2}^+$.
\end{lemma}

\begin{proof}
Since $u\geq 0$, $u(e_n/2) \geq 1$ and $\|f\|_{L^q(B_1^+)}\leq \delta_0$ for some $q>n/p$, by the Harnack inequality (see \cite[Theorem 3]{MR170096} or \cite{MR226198, MR3931688}), there exists a constant $c_0 > 0$ (depending only on $n, p, q$) such that
\begin{equation*}
  u \geq c_0~~\mbox{ in }~~\overline{B}_{1/4}(e_n/2).
\end{equation*}

Consider the annular domain $\Omega = B_{1/2}(e_n/2) \setminus \overline{B}_{1/4}(e_n/2)$. We define the barrier function:
\begin{equation*}
  v(x) =\mu\left(e^{-\alpha |x - e_n/2|^2} - e^{-\alpha/4}\right),
\end{equation*}
where $\alpha, \mu>0$ are to be chosen later.
A direct computation shows that, by choosing $\alpha$ large enough (depending only on $n$ and $p$),
and then choosing $\mu > 0$ sufficiently small, we have
\begin{equation*}
  \begin{cases}
  \mathrm{div}(|Dv|^{p-2}Dv)
  \geq 0 &~~ \mbox{ in }~~\Omega,\\
  v\leq c_0 &~~\mbox{ on }~~\partial B_{1/4}(e_n/2),\\
  v= 0 &~~\mbox{ on }~~\partial B_{1/2}(e_n/2).
  \end{cases}
\end{equation*}
By \Cref{ABPw}, we have
\begin{equation*}
  u(x) \geq v(x) - C\|f\|_{L^q(B_1^+)}^{1/(p-1)} \quad \mbox{ in } \Omega,
\end{equation*}
where $C > 0$ depends only on $n, p, q$.

By a direct calculation,
\begin{equation*}
  v(x)\geq cx_n \quad ~~\mbox{ on }~~\{x'=0,0\leq x_n\leq 1/4\}.
\end{equation*}
Then, we obtain
\begin{equation*}
  u(x) \geq c x_n - C\|f\|_{L^q(B_1^+)}^{1/(p-1)} \quad \mbox{ on }~~\{x'=0,0\leq x_n\leq 1/4\}.
\end{equation*}
For any $x'_0 \in T_{1/2}$, by a similar argument as above, we have
\begin{equation*}
 u(x) \geq c x_n - C\|f\|_{L^q(B_1^+)}^{1/(p-1)} \quad \mbox{ on }~~\{x'=x'_0,0\leq x_n\leq 1/4\}.
\end{equation*}

Finally, by combining with the Harnack inequality,
\begin{equation*}
  u(x) \geq c x_n - C\|f\|_{L^q(B_1^+)}^{1/(p-1)} \quad \mbox{in } B^+_{1/2}.
\end{equation*}
\end{proof}

\medskip

We now prove \Cref{t.Hopf} by the iteration scheme developed in \cite{MR4713521,DLL_2025}.


\begin{proof}
First, let us make some normalization. The normalization argument is similar to that in \cite[Step 1 in Section 3.2]{DLL_2025}.
Let
\begin{equation*}
	\omega(r)=\max(\omega_{\Omega}(r),\hat{C}\omega_f^{1/(p-1)}(r)),\quad \forall ~0<r<1,
\end{equation*}
where $\hat{C}>0$ will be specified later. Without loss of generality, we assume
\begin{equation}\label{e2.2}
	u(e_n/2)\geq 1, \quad \omega(1)\leq \eta^2,\quad \int_0^{1}\frac{\omega(r)}{r}dr\leq \eta^2,
\end{equation}
where $0<\eta<1/8$ will be chosen later.

Otherwise, we normalize the problem as follows. Since $\Omega$ satisfies the interior $C^{1,\mathrm{Dini}}$ condition at $0\in\partial\Omega$, there exists $0<r_1<1/4$ depending on $n,p,q$ and $\omega_\Omega$ such that
\begin{equation}\label{e2.5}
	\omega_{\Omega}(r_1)\leq \eta^2/2,\quad \int_0^{r_1}\frac{\omega_{\Omega}(r)}{r}dr\leq \eta^2/2.
\end{equation}
Let $x_0=r_1 e_n/2$.
By the Harnack inequality (see \cite[Theorem 5]{MR170096} or \cite{MR226198}) and the assumption
\begin{equation}\label{epsilon0}
  \|f\|^{1/(p-1)}_{C_q^{-1,\text{Dini}}(0)}\leq \varepsilon_0 u(e_n/2),
\end{equation}
there exists a constant $0<c<1$, depending on $n,p,q$ and $\omega_\Omega$, such that
\begin{equation}\label{nor}
	u\left(x_0\right)\geq c u\left( e_n/2\right)-C\varepsilon_0  u\left( e_n/2\right)
	\geq \frac{c}{2}u\left( e_n/2\right)
\end{equation}
by taking $\varepsilon_0\leq c/(2C)$.

Now define
\begin{equation*}
	y=\frac{x}{r_1},  \quad
	\tilde{u}(y)=\frac{u(x)}{u(x_0)}.
\end{equation*}
Then
\begin{equation*}
	\mathrm{div}(|D \tilde{u}|^{p-2}D \tilde{u})=\tilde{f}\quad  ~~\mbox{ in }~~ \widetilde \Omega_1,
\end{equation*}
where
\begin{equation*}
	\tilde f(y)=\frac{r_1^p f(x)}{(u(x_0))^{p-1}}, \quad
	\widetilde \Omega=\frac{\Omega}{r_1}.
\end{equation*}
Clearly, we have
\begin{equation}\label{e2.1}
	\tilde{u}(e_n/2)=1.
\end{equation}
For any $0<r<1$,
\begin{equation*}
	\begin{aligned}
		\|\tilde{f}\|^*_{L^{q}(\tilde{\Omega}_r)}
		=\frac{r_1^{p}}{(u(x_0))^{p-1}}\|f\|^*_{L^{q}(\Omega_{r_1r})}
		\leq \frac{r_1^{p}}{(u(x_0))^{p-1}}
		(r_1 r)^{-1}
		\omega_f(r_1 r)
		=r^{-1}\cdot \frac{r_1^{p-1}}{(u(x_0))^{p-1}} \omega_f(r_1 r).
	\end{aligned}
\end{equation*}
By setting
\begin{equation*}
  \omega_{\tilde f}(r)=\frac{r_1^{p-1}}{(u(x_0))^{p-1}} \omega_f(r_1 r),
\end{equation*}
we have
\begin{equation*}
  \|\tilde{f}\|^*_{L^{q}(\tilde{\Omega}_r)}\leq r^{-1} \omega_{\tilde f}(r).
\end{equation*}

Thus, by combining with \eqref{epsilon0} and \eqref{nor}, we have
\begin{equation}\label{e2.4-0}
	\hat{C}\omega_{\tilde{f}}^{1/(p-1)}(1)=\frac{\hat{C}r_1}{u(x_0)}\omega_f^{1/(p-1)}(r_1)
	\leq \frac{\hat{C} r_1 \varepsilon_0u(e_n/2)}{c/2 \cdot u(e_n/2)}\leq \eta^2/2
\end{equation}
by taking $\varepsilon_0$ small enough. Similarly,
\begin{equation}\label{e2.4}
	\int_{0}^{1} \frac{\hat{C}\omega_{\tilde{f}}^{1/(p-1)}(r)}{r}dr=
	\frac{\hat{C}r_1}{u(x_0)}\int_{0}^{r_1} \frac{\omega_{f}^{1/(p-1)}(s)}{s}ds
	\leq \frac{\hat{C} r_1 \varepsilon_0u(e_n/2)}{c/2 \cdot u(e_n/2)}
    \leq \eta^2/2.
\end{equation}
Finally, let us consider $\tilde{\Omega}$. Define
\begin{equation*}
	\omega_{\tilde{\Omega}}(r)=\omega_{\Omega}(r_1r),\quad \forall ~0<r<1.
\end{equation*}
By noting \eqref{e2.5},
\begin{equation}\label{e2.6}
	\omega_{\tilde{\Omega}}(1)=\omega_{\Omega}(r_1)\leq \eta^2/2, \quad
	\int_{0}^{1} \frac{\omega_{\tilde{\Omega}}(r)}{r}dr
	=\int_{0}^{r_1}\frac{\omega_{\Omega}(s)}{s}ds\leq \eta^2/2.
\end{equation}
Consequently, the normalized quantities $\tilde{u}$, $\tilde{f}$ and $\tilde{\Omega}$ satisfy the assumptions in \eqref{e2.2}. Therefore, throughout the proof, we assume that \eqref{e2.2} holds for $u$, $f$ and $\Omega$.

\medskip

Secondly, we show the following claim:
there exist two constants $\bar C\geq 1$, $\tilde a$ and a nonnegative sequence $a_k$ $(k\geq -1)$ such that for any $k=0,1,2,...$,
we have
\begin{equation}\label{3.7'}
	\inf_{\Omega^+_{\eta^{k+1}}}(u-\tilde a x_n+a_k x_n)\geq - \eta^k A_k,
\end{equation}
\begin{equation}\label{3.8'}
	a_k-a_{k-1}\leq \bar C  A_k
\end{equation}
and
\begin{equation}\label{3.81'}
	a_k\leq \tilde a/2,
\end{equation}
where
\begin{equation*}
	A_0=\eta^2,\quad A_{k+1}=\max\{\omega(\eta^k),\eta^{\bar \alpha/2}A_{k}\}
\end{equation*}
and $\bar \alpha$ is defined as in \Cref{t.p-Lap.flat} and $\bar C\geq 1$, $\tilde a$ depend only on $n,p$ and $q$.

\medskip

For $k=0$, the interior $C^{1,\mathrm{Dini}}$ condition gives
\begin{equation*}
  B_{1/2}^++\omega(1)e_n\subset\Omega_1.
\end{equation*}
By the Harnack inequality, $u((1/4+\omega(1))e_n)\geq c_0>0$,
where $c_0$ depends only on $n,p,q$, provided that $\hat C$ is sufficiently large.
Applying \Cref{hpflat} after translation, scaling and division by $c_0$,
there exists $0<c_1<1/2$, depending only on $n,p,q$, such that
\begin{equation*}
  u(x)\geq c_1(x_n-\omega(1))-C\|f\|_{L^q(\Omega_1)}^{1/(p-1)}
\geq c_1 x_n-\frac{3}{2}c_1\omega(1)
\quad\mathrm{in}\ \ B^+_{1/4}+\omega(1)e_n.
\end{equation*}
Here $\hat C$ is chosen so large that
\begin{equation*}
  C\|f\|_{L^q(\Omega_1)}^{1/(p-1)}
\leq C\omega_f^{1/(p-1)}(1)
\leq\frac{c_1}{2}\omega(1).
\end{equation*}
Set $\eta_1=1/8$. Since
\begin{equation*}
  B_{\eta_1}\cap\{x_n>\omega(1)\}
\subset B_{1/4}^++\omega(1)e_n
\end{equation*}
and $u\geq0$, we have
\begin{equation*}
  u(x)\geq c_1 x_n-\frac{3}{2}c_1\omega(1) \quad \mathrm{in}\ \ \Omega^+_{\eta_1}.
\end{equation*}
Then by setting
\begin{equation*}
  \tilde a= c_1,\quad a_{-1}=a_0=0
\end{equation*}
and $\eta\leq \eta_1$, we deduce that
\begin{equation*}
  \inf_{\Omega_\eta^+}(u- c_1 x_n)
\geq \inf_{\Omega_{\eta_1}^+}(u- c_1 x_n)
\geq -\frac{3}{2}c_1 \omega(1)\geq -\frac{3}{2}c_1 \eta^2\geq -\eta^2.
\end{equation*}
Hence \eqref{3.7'}, \eqref{3.8'} and \eqref{3.81'} hold for $k=0$.

Suppose that they hold for $k$. We prove that they also hold for $k+1$. Let $r=\eta^{k+1}$ and $v$ be a viscosity solution of
\begin{equation}\label{v'}
	\left\{
	\begin{aligned}
		&\mathrm{div}(|Dv+\tilde a e_n-a_k e_n|^{p-2}(Dv+\tilde a e_n-a_k e_n))= 0  &\mathrm{in} \ \   &  B_r^+,\\
		&v=0\   &\mathrm{on}   \ \ & T_r,\\
		&v= -\eta^k A_k     &\mathrm{on}  \ \  & \partial B_r^+\backslash T_r.
	\end{aligned}
	\right.
\end{equation}
Note that
\begin{equation*}
  v\in \mathcal{S}(1,p-1,0).
\end{equation*}
By \Cref{t.p-Lap.flat}, there exists a constant $a_r$ such that
\begin{equation}\label{vh2}
	0\leq a_r\leq \frac{C_1}{r}\|v\|_{L^\infty(B_{r}^+)}
	\leq \frac{C_1}{r}\eta^k A_k
	= \frac{C_1}{\eta} A_k
\end{equation}
and
\begin{equation}\label{vh3}
	\begin{aligned}
		\|v+a_r x_n\|_{L^\infty(B_{\eta^{k+2}}^+)}	
		=&\|v+a_r x_n\|_{L^\infty(B_{\eta r}^+)}
		\leq \frac{C_1}{r^{1+\bar \alpha}}(\eta r)^{1+\bar \alpha}\|v\|_{L^\infty(B_{r}^+)}\\
		\leq& C_1\eta^{1+\bar \alpha}\eta^k A_k
		\leq C_1 \eta^{\bar \alpha/2} \eta^{k+1} A_{k+1},
	\end{aligned}
\end{equation}
where $\bar\alpha$ and $C_1$ are as in \Cref{t.p-Lap.flat}.

\medskip

Set
\begin{equation*}
  \hat u=u-\tilde a x_n+a_k x_n.
\end{equation*}
It then follows from $u\geq 0$ and $a_k\geq 0$ that  $\hat{u}$ satisfies
\begin{equation}\label{tilde_u-hopf}
	\begin{cases}
		\mathrm{div} (|D\hat{u}+\tilde a e_n-a_k e_n|^{p-2}(D\hat{u}+\tilde a e_n-a_k e_n))= f  &~~\mbox{ in }~~  \Omega_r^+,\\
		\hat{u}\geq -\tilde a x_n   &~~\mbox{ on }~~     \partial\Omega_r^+\cap ( B_r\cup T_r),\\
		\hat{u}\geq -\eta^k A_k   &~~\mbox{ on }~~  \partial\Omega_r^+\cap (\partial B_r^+\backslash T_r).
	\end{cases}
\end{equation}
By Lemma \ref{ABPw}, we deduce that
\begin{equation*}
  \inf_{\Omega^+_{\eta^{k+2}}} (\hat u-v)
\geq \inf_{\Omega^+_{r}} (\hat u-v)
\geq -\tilde a r\omega(r)-C_2 r^{\frac{p}{p-1}-\frac{n}{q(p-1)}}\|f\|^{\frac{1}{p-1}}_{L^{q}(\Omega_{r}^+)},
\end{equation*}
where $C_2$ depends on $n,p$ and $q$. We compute that
\begin{equation*}
  \tilde a r\omega(r)= c_1 \eta^{k+1}\omega(\eta^{k+1})
\leq c_1 \eta^{k+1}A_{k+1}
\leq \frac{1}{2}\eta^{k+1}A_{k+1}.
\end{equation*}
By taking $\hat{C}$ large enough,
\begin{equation*}
	\begin{aligned}
		C_2 r^{\frac{p}{p-1}-\frac{n}{q(p-1)}}\|f\|^{\frac{1}{p-1}}_{L^{q}(\Omega_r^+)}
		\leq \frac{1}{4}\hat{C}r\omega_f^{1/(p-1)}(r)\leq\frac{1}{4} r\omega(r)\leq \frac{1}{4} \eta^{k+1}A_{k+1}.
	\end{aligned}
\end{equation*}
Then,
\begin{equation}\label{wh}
	\inf_{\Omega^+_{\eta^{k+2}}} (\hat u-v)
	\geq -\frac{3}{4}\eta^{k+1}A_{k+1}.
\end{equation}

Set
\begin{equation*}
  a_{k+1}=a_k+a_r.
\end{equation*}
By taking $\eta$ small enough such that
\begin{equation*}
  \eta\leq \eta_1, \quad C_1\eta^{\bar\alpha/2}\leq 1/4,
\end{equation*}
we then derive from \eqref{vh3} that
\begin{equation*}
  	\|v+a_r x_n\|_{L^\infty(B_{\eta^{k+2}}^+)}	
\leq C_1 \eta^{\bar \alpha/2} \eta^{k+1} A_{k+1}\leq 1/4\eta^{k+1} A_{k+1}.
\end{equation*}
It then follows from the above estimate and \eqref{wh} that
\begin{equation*}
	\begin{aligned}
		\inf_{\Omega_{\eta^{k+2}}^+}(u-\tilde a x_n+a_{k+1}x_n)
		&=\inf_{\Omega_{\eta^{k+2}}^+}(u-\tilde a x_n+a_{k}x_n-v+v+a_r x_n)\\
		&=\inf_{\Omega_{\eta^{k+2}}^+}(\hat u-v+v+a_r x_n)
		\geq  -\eta^{k+1} A_{k+1}.
	\end{aligned}
\end{equation*}
In addition, from \eqref{vh2},
\begin{equation*}
  a_{k+1}-a_k=a_r
\leq \frac{C_1}{\eta} A_{k}
\leq \frac{C_1}{\eta^{1+\frac{\bar \alpha}{2}}} A_{k+1}
:=\bar C A_{k+1}.
\end{equation*}
Recall that $\sum\limits_{i=0}^{\infty} A_i\leq 4\eta^2$ (see \eqref{Ak}).
Then by choosing $\eta$ small enough we have
\begin{equation*}
  a_{k+1}\leq \sum\limits_{i=0}^{k+1}|a_i-a_{i-1}|
\leq \bar C \sum\limits_{i=0}^{\infty} A_i\leq 4C_1 \eta^{1-\frac{\bar \alpha}{2}}\leq \tilde a/2.
\end{equation*}
Consequently, \eqref{3.7'}, \eqref{3.8'} and \eqref{3.81'} hold for $k+1$. By induction, the claim is proved.

By an argument similar to that in \cite[Step 3]{DLL_2025}, we obtain
\begin{equation*}
  u(rl)\geq cl_n r,\quad \forall 0<r<\eta,
\end{equation*}
for every unit vector $l$ with $l_n>0$.

\end{proof}

\section{Boundary pointwise $C^1$ regularity}
\label{sec:C1}
In this section, we prove \Cref{t.C1}. First, we establish the following key lemma.

\begin{lemma}\label{k1}
Let $\xi\in \mathbb{R}^n$, and let $\bar\alpha,\bar C$ be as in \Cref{t.p-Lap.flat}. For any $0<\alpha<\bar\alpha$, there exists $\delta>0$ such that if $u$ satisfies
\begin{equation*}
	\begin{cases}
	\mathrm{div}(|Du+\xi|^{p-2}(Du+\xi))  =f   &~~\mbox{ in }~~\ \Omega_1,\\
	u=g   &~~\mbox{ on }~~(\partial\Omega)_1,
	\end{cases}
\end{equation*}
with $\|u\|_{L^\infty(\Omega_1)}\leq 1$, $\|f\|_{L^\infty(\Omega_1)}^{1/(p-1)}\leq\delta$, $\|g\|_{L^\infty((\partial\Omega)_1)}\leq\delta$ and $\operatorname{osc}_{B_1}\partial\Omega\leq\delta$,
then there exists a constant $a$ such that
\begin{equation*}
  \|u-a x_n\|_{L^\infty(\Omega_{\eta})}\leq \eta^{1+\alpha}
\end{equation*}
and
\begin{equation*}
  |a|\leq \bar C,
\end{equation*}
where $\eta$ depends only on $n$, $p$ and $\alpha$.
\end{lemma}

\begin{proof}
We argue by contradiction. Suppose that the conclusion fails. Then there exist $0<\alpha<\bar{\alpha}$ and sequences $\{\xi_k\}$, $\{u_k\}$, $\{f_k\}$, $\{g_k\}$ and $\{\Omega_k\}$ such that $u_k$ is a viscosity solution of
\begin{equation}\label{uk}
	\begin{cases}
     \mathrm{div}(|Du_k+\xi_k|^{p-2}(Du_k+\xi_k)) =f_k ~~&\mbox{ in }~~\ (\Omega_k)_1,\\
	 u_k=g_k   &~~\mbox{ on }~~\ (\partial\Omega_k)_1
	\end{cases}
\end{equation}
satisfying
\begin{equation*}
  \|u_k\|_{L^\infty((\Omega_k)_1)}\leq 1,
\end{equation*}
and
\begin{equation*}
	\|f_k\|_{L^\infty((\Omega_k)_1)}^{1/(p-1)}\leq1/k,\ \|g_k\|_{L^\infty((\partial\Omega_k)_1)}\leq1/k,\
		\operatorname{osc}_{B_1}\partial\Omega_k\leq1/k,
\end{equation*}
while
\begin{equation}\label{3.7}
		\|u_k-a x_n\|_{L^\infty((\Omega_k)_{\eta})}>\eta^{1+\alpha},\ \ \forall |a|\leq \bar{C}
\end{equation}
where $0<\eta<1$ is taken small so that
\begin{equation}\label{barC}
  \bar C \eta^{\bar \alpha-\alpha}<1/2.
\end{equation}
	
Clearly, the functions $u_k$ are uniformly bounded. By \Cref{l3.3}, the functions $u_k$ are equicontinuous in any compact subset in $B_1^+$. Hence, there exist a subsequence of $\{u_k\}$ (still denoted by $u_k$) and a function $u_0$ such that
\begin{equation*}
  u_k\rightarrow u_0 \ \ ~~\mbox{ uniformly in compact sets of }~~  B_{1/2}^+.
\end{equation*}

In the following, we will show that
\begin{equation}\label{e2.3}
u_0\in \mathcal{S}(1,p-1,0)~~\mbox{ in }~~B_{1/2}^+.
\end{equation}
We only prove $\mathcal{M}^-(D^2 u_0,1,p-1)\leq 0$ in the viscosity sense. As in the proof of \Cref{l.vis.Pucci}, we construct a test function $\varphi$ touching $u_0$ from below at $x_0\in B^+_{1/2}$ and satisfying $\varphi <u_0$ in $B^+_{1/2}\backslash \{x_0\}$. Since $u_k$ converges to $u_0$ uniformly in any compact subset of $B^+_{1/2}$, $\varphi+c_k$ touches $u_k$ from below at some point $x_k\in B^+_{1/2}$ where $c_k\to 0$ and $x_k\to x_0$. By the definition of viscosity solution for $u_k$, whenever $D\varphi(x_k)+\xi_k\neq 0$,
\begin{equation*}
  \left|D\varphi(x_k)+\xi_k \right|^{p-2}\left(\operatorname{tr} D^2\varphi(x_k)+(p-2)\frac{\left(D\varphi(x_k)+\xi_k\right)_i
  \left(D\varphi(x_k)+\xi_k\right)_j}{|D\varphi(x_k)+\xi_k|^{2}}\varphi_{ij}(x_k)\right)\leq f_k(x_k),
\end{equation*}

If we can choose a subsequence (still denoted by $x_k,\xi_k$) such that
\begin{equation*}
\liminf_{k\to \infty}|D\varphi(x_k)+\xi_k |>0,
\end{equation*}
we have
\begin{equation*}
  \begin{aligned}
\mathcal{M}^-(D^2 \varphi(x_0),1,p-1)=&\lim_{k\to \infty}\mathcal{M}^-(D^2 \varphi(x_k),1,p-1)\\
    \leq& \liminf_{k\to \infty}\left(\operatorname{tr} D^2\varphi(x_k)+(p-2)\frac{\left(D\varphi(x_k)+\xi_k\right)_i
  \left(D\varphi(x_k)+\xi_k\right)_j}{|D\varphi(x_k)+\xi_k|^{2}}\varphi_{ij}(x_k)\right)\\
  \leq& \liminf_{k\to \infty}|f_k(x_k)||D\varphi(x_k)+\xi_k|^{2-p}
  =0.
  \end{aligned}
\end{equation*}

Otherwise, i.e.,
\begin{equation*}
\lim_{k\to \infty}|D\varphi(x_k)+\xi_k |=0,
\end{equation*}
we can prove by contradiction as in the proof of \Cref{l.vis.Pucci}. Assume that $\mathcal{M}^-(D^2\varphi(x_0),1,p-1)>0$. Similar to the proof of \Cref{l.vis.Pucci}, we can construct $\psi$, a perturbation of $\varphi$, such that $\psi$ touches $u_0$ at $\tilde x_0\in B^+_{1/2}$ and
\begin{equation*}
\liminf_{k\to \infty}|D\psi (x_k)+\xi_k |>0,
\end{equation*}
Then, we can prove $\mathcal{M}^-(D^2 \varphi(x_0),1,p-1)\leq 0$ as well.

Next, we prove $u_0=0$ on $T_{1/2}$.
By \Cref{d1},
\begin{equation*}
  |u_k (x)| \leq C(x_n+2/k),\ \forall x \in \Omega_k \cap B_{1/2}.
\end{equation*}
For any $x\in B^+_{1/2}$, by taking $k\rightarrow \infty$, we have
\begin{equation*}
  |u_0(x)|\leq Cx_n.
\end{equation*}
Therefore, $u_0$ is continuous up to $T_{1/2}$ and $u_0\equiv 0$ on $T_{1/2}$.

By the boundary $C^{1,\alpha}$ regularity for fully nonlinear equations (see \cite[Lemma 3.1]{MR3246039}), there exist $\bar\alpha,\bar{C}$ (depending on $n,p$) and $\bar a$ with $|\bar a|\leq \bar C$ such that
\begin{equation*}
	|u_0-\bar a x_n|\leq \bar{C}\eta^{1+\bar\alpha}\leq \eta^{1+\alpha}/2  \ \ ~~\mbox{ in }~~ \ {B_{\eta}^+},
\end{equation*}
where \eqref{barC} is used in the last inequality.
	
Letting $k\rightarrow \infty$ in \eqref{3.7} with $a=\bar a$, we deduce that
\begin{equation*}
  \|u_0-\bar a x_n\|_{L^{\infty}(B_{\eta}^+)}\geq \eta^{1+\alpha}.
\end{equation*}
Consequently, we get a contradiction.
\end{proof}

We are now in a position to prove \Cref{t.C1}.

\begin{proof}
We first make some normalizations. Let $\delta$ and $\bar C$ be as in \Cref{k1}. Set
\begin{equation*}
  M=\|u\|_{L^\infty(\Omega_1)}+\delta^{-1}\left(
\|f\|_{L^\infty(\Omega_1)}^{1/(p-1)}+2\|g\|_{C^{1,\mathrm{Dini}}(0)}\right).
\end{equation*}
Define
\begin{equation*}
  z=\frac{x}{r_0},\qquad \widetilde\Omega=\frac{\Omega}{r_0},\qquad
\tilde u(z)=\frac{u(x)-g(0)-Dg(0)\cdot x}{M},
\end{equation*}
where $0<r_0<1$ is chosen sufficiently small, depending on $n,p$ and $\omega_\Omega$.
Then $\tilde u$ satisfies
\begin{equation*}
  \begin{cases}
\mathrm{div}\left(|D\tilde u+\xi_0|^{p-2}(D\tilde u+\xi_0)\right)=\tilde f
&\mbox{in }\ \ \ \ \widetilde\Omega_1,\\
\tilde u=\tilde g&\mbox{on }\ \ (\partial\widetilde\Omega)_1,
\end{cases}
\end{equation*}
where
\begin{equation*}
  \xi_0=\frac{r_0Dg(0)}M,\qquad
\tilde f(z)=\frac{r_0^p f(x)}{M^{p-1}},\qquad
\tilde g(z)=\frac{g(x)-g(0)-Dg(0)\cdot x}{M}.
\end{equation*}
The corresponding moduli are
\begin{equation*}
  \omega_{\tilde g}(r)=\frac{r_0}{M}\omega_g(r_0r),\qquad
\omega_{\widetilde\Omega}(r)=\omega_\Omega(r_0r).
\end{equation*}
Hence,
\begin{equation*}
  \tilde u(0)=\tilde g(0)=0\quad \mbox{and} \quad D\tilde g(0)=0,
\end{equation*}
and
\begin{equation}\label{eq:c1-smallness}
\begin{gathered}
\|\tilde u\|_{L^\infty(\widetilde\Omega_1)}\leq 1,\qquad
\|\tilde f\|_{L^\infty(\widetilde\Omega_1)}^{1/(p-1)}\leq \delta,\\
\|\tilde g\|_{C^{1,\mathrm{Dini}}(0)}\leq \delta/2,\qquad
\|\partial \widetilde\Omega \|_{C^{1,\mathrm{Dini}}(0)}\leq \delta/(8\bar C).
\end{gathered}
\end{equation}
The last inequality can be achieved by choosing $r_0$ sufficiently small.
Finally, we define
\begin{equation}\label{eq:c1-modulus}
\omega(r)=\max\left\{
\frac{\omega_{\tilde g}(r)}{\|\tilde g\|_{C^{1,\mathrm{Dini}}(0)}},
\frac{\omega_{\widetilde\Omega}(r)}{\|\partial \widetilde\Omega\|_{C^{1,\mathrm{Dini}}(0)}}
\right\}.
\end{equation}

To prove that $\tilde u$ is $C^1$ at $0$, it suffices to establish the following claim by induction: there exist $0<\eta<1$ and a sequence $\{a_k\}$, $k\geq -1$, such that, for every $k\geq 0$,
\begin{equation}\label{eq:c1-induction-error}
\|\tilde u-a_k z_n \|_{L^\infty(\widetilde\Omega_{\eta^k})}\leq \eta^k A_k
\end{equation}
and
\begin{equation}\label{eq:c1-induction-slope}
|a_k-a_{k-1}|\le\bar C A_{k-1},
\end{equation}
where
\begin{equation}\label{eq:c1-A}
A_{-1}=A_0=1,\qquad
A_k=\max\{\omega(\eta^k),\eta^{\bar\alpha/2}A_{k-1},\eta^{k/(p-1)}\}
\quad(k\geq 1),
\end{equation}
where $\bar\alpha$ and $\bar C$ are the constants in \Cref{k1}, and $\eta$ depends only on $n$ and $p$.

For $k=0$, the conclusion is immediate by taking $a_{-1}=a_0=0$. Assume that the claim holds for some $k\geq 0$. Define
\begin{equation*}
  y=\frac{z}{\eta^k},\qquad
v(y)=\frac{\tilde u(z)-a_k z_n}{\eta^k A_k},\qquad
\hat\Omega=\frac{\widetilde\Omega}{\eta^k}.
\end{equation*}
Then $v$ satisfies
\begin{equation}\label{eq:c1-rescaled-equation}
\begin{cases}
\displaystyle\mathrm{div}\left(
\left|Dv+\frac{a_ke_n+\xi_0}{A_k}\right|^{p-2}
\left(Dv+\frac{a_ke_n+\xi_0}{A_k}\right)\right)=\hat f
&\mbox {in } \hat\Omega_1,\\
v=\hat g&\mbox{on }(\partial\hat\Omega)_1,
\end{cases}
\end{equation}
where
\begin{equation*}
  \hat f(y)=\frac{\eta^k \tilde f(z)}{A_k^{p-1}},\qquad
\hat g(y)=\frac{\tilde g(z)-a_k z_n}{\eta^k A_k}.
\end{equation*}
By the induction hypothesis \eqref{eq:c1-induction-error},
\begin{equation*}
  \|v\|_{L^\infty(\hat\Omega_1)}\leq 1.
\end{equation*}
By \eqref{eq:c1-A}, we have $\eta^{k/(p-1)}\leq A_k$. Then
\begin{equation*}
  \|\hat f\|_{L^\infty(\hat\Omega_1)}^{1/(p-1)}
=\frac{\eta^{k/(p-1)}\|\tilde f\|_{L^\infty(\widetilde\Omega_{\eta^k})}^{1/(p-1)}}{A_k}
\le\frac{\eta^{k/(p-1)}\delta}{A_k}\le\delta.
\end{equation*}
Similarly,
\begin{equation*}
  \operatorname{osc}_{B_1}\partial \hat\Omega
=\frac1{\eta^k}\operatorname{osc}_{B_{\eta^k}}\partial \widetilde\Omega \leq \delta.
\end{equation*}

On the other hand, using the definition of $A_k$, one can show that
\begin{equation*}
  \sum_{i=0}^\infty A_i\leq 4
\end{equation*}
by choosing $\eta$ sufficiently small (see \cite[(3.15) on p. 11]{DLL_2025} for the proof).
Therefore,
\begin{equation*}
  |a_k|\le\sum_{i=1}^k|a_i-a_{i-1}|
\leq \bar C\sum_{i=0}^\infty A_i\leq 4\bar C.
\end{equation*}
Combining this estimate with the assumptions in \eqref{eq:c1-smallness}, we obtain
\begin{equation}\label{eq:c1-boundary-smallness}
\begin{split}
\|\hat g\|_{L^\infty((\partial \hat\Omega)_1)}
&=\frac1{\eta^k A_k}\|\tilde g(z)-a_k z_n\|_{L^\infty((\partial \widetilde\Omega)_{\eta^k})}\\
&\leq \frac1{\eta^k A_k}\left(
\|\tilde g(z)\|_{L^\infty((\partial \widetilde\Omega)_{\eta^k})}
+|a_k|\|z_n\|_{L^\infty((\partial \widetilde\Omega)_{\eta^k})}\right)\\
&\leq \frac{\eta^k\omega(\eta^k)}{\eta^k A_k}
\left(\|\tilde g(z)\|_{C^{1,\mathrm{Dini}}(0)}
+4\bar C\|\partial\widetilde\Omega\|_{C^{1,\mathrm{Dini}}(0)}\right)\le\delta.
\end{split}
\end{equation}
Hence, $v(y)$, $\hat f(y)$, $\hat g(y)$ and $\hat\Omega$ satisfy the assumptions of \Cref{k1}. Applying \Cref{k1} with $\alpha=\bar\alpha/2$, there exists a constant $\bar a$ such that
\begin{equation*}
  \|v-\bar a y_n\|_{L^\infty(\hat\Omega_\eta)}\leq \eta^{1+\bar\alpha/2}
\end{equation*}
and
\begin{equation*}
  |\bar a|\leq \bar C.
\end{equation*}
Equivalently,
\begin{equation*}
  \left\|\frac{\tilde u(z)-a_k z_n}{\eta^k A_k}
-\bar a\frac{z_n}{\eta^k}\right\|_{L^\infty(\widetilde\Omega_{\eta^{k+1}})}
\leq \eta^{1+\bar\alpha/2}.
\end{equation*}
By setting $a_{k+1}=a_k+\bar a A_k$, we have
\begin{equation*}
  |a_{k+1}-a_k|\leq |\bar a|A_k\leq \bar C A_k
\end{equation*}
and
\begin{equation*}
  \|\tilde u-a_{k+1} z_n\|_{L^\infty(\widetilde\Omega_{\eta^{k+1}})}
\leq \eta^k A_k\eta^{1+\bar\alpha/2}\le\eta^{k+1}A_{k+1}.
\end{equation*}
Thus, the induction is complete.

The conclusion $u \in C^1(0)$ then follows by an argument similar to that in \cite[Lemma 3.16]{lian2020pointwise}.
\end{proof}


\medskip

\printbibliography

\end{document}